\documentclass[12pt]{amsart}
\usepackage{amssymb}
\usepackage{amsmath}
\usepackage{latexsym}
\def\eps{\varepsilon}
\def\R{\mathbb R}
\def\C{\mathbb C}
\def\RR{\overline{{\mathbb R}^n}}
\def\N{\mathbb N}
\def\Z{\mathbb Z}
\def\dist{\operatorname{dist}}
\def\card{\operatorname{card}}
\def\capacity{\operatorname{cap}}
\def\deg{\operatorname{deg}}
\def\dim{\operatorname{dim}}
\def\diam{\operatorname{diam}}
\def\Real{\operatorname{Re}}
\def\Imag{\operatorname{Im}}
\def\qr{quasi\-regular }
\newtheorem{lemma}{Lemma}[section]
\newtheorem{theorem}{Theorem}[section]
\newtheorem{cor}{Corollary}[section]

\theoremstyle{definition}
\newtheorem{definition}{Definition}[section]

\newtheorem{example}{Example}[section]
\theoremstyle{remark}
\newtheorem{remark}{Remark}[section]

\numberwithin{equation}{section}
\begin{document}

\title[Iteration of entire quasiregular maps]{Foundations for an iteration theory of entire quasiregular maps}

\author{Walter Bergweiler}
\address{Mathematisches Seminar,
Christian--Albrechts--Universit\"at zu Kiel,
Lude\-wig--Meyn--Str.~4,
D--24098 Kiel,
Germany}
\email{bergweiler@math.uni-kiel.de}
\author{Daniel A. Nicks}
\address{School of Mathematical Sciences, University of Nottingham, Nottingham NG7 2RD, UK}
\email{Dan.Nicks@nottingham.ac.uk}
\thanks{The first author is
supported by the Deutsche Forschungsgemeinschaft, Be 1508/7-2 and
the ESF Networking Programme HCAA \\ \emph{Mathematics Subject Classification:} Primary 37F10; Secondary 30C65, 30D05}

\begin{abstract}
The Fatou-Julia iteration theory of rational functions has
been extended to quasiregular mappings in higher dimension
by various authors.
The purpose of this paper is an analogous extension
of the iteration theory of transcendental entire functions.
Here the Julia set is defined as the set of all points such that complement of
the forward orbit of any neighbourhood has capacity zero.
It is shown that for maps which are not of polynomial type the Julia
set is non-empty and has many properties of the classical Julia
set of transcendental entire functions.
\end{abstract}
\maketitle

\section{Introduction and main results}\label{Introduction}
In 1918-20, Fatou~\cite{Fatou19} and Julia~\cite{Julia18} wrote long memoirs on the iteration of rational functions and thereby created the field now known as complex dynamics. The analogies (and differences) that arise in the corresponding theory for transcendental entire functions were studied by Fatou~\cite{Fatou26} in 1926.

Many of the results of the Fatou-Julia theory for rational functions, considered as self-maps of the Riemann sphere $S^2$, have been extended to uniformly quasiregular self-maps of the $n$-sphere $S^n$ where $n\ge 2$ by Hinkkanen, Iwaniec, Martin, Mayer and others; see \cite[Chapter~21]{Iwaniec01}, \cite[Chapter~4]{Siebert04},  and \cite[Section~4]{Bergweiler10}
 for surveys. Here a quasiregular map $f\colon S^n\to S^n$ is called \emph{uniformly quasiregular} if there exists a uniform bound on the dilatation of the iterates $f^k$ of~$f$. (We will recall the definition of quasiregularity, in particular the notions of dilatation and inner dilatation, in section~\ref{quasiregular}.) In principle, it would also be possible to extend some of Fatou's results about transcendental entire functions $f\colon \C\to\C$ to uniformly quasiregular maps $f\colon \R^n\to\R^n$ where $n\ge 2$. However, for $n\ge 3$ no examples of such maps with an essential singularity at $\infty$ are known yet.

On the other hand, the iteration of quasiregular analogues of the exponential function (called the Zorich map) and the trigonometric functions were studied in~\cite{Bergweiler10a,Bergweiler11,FletcherSine}. Also~\cite{Bergweiler09} contains a general result about the \emph{escaping set}
\[I(f)=\{x\in \R^n:f^k(x)\to \infty \text{ as } k\to \infty\}\]
of a quasiregular map $f\colon \R^n\to\R^n$.

However, no systematic theory in the spirit of Fatou and Julia has been developed yet for quasiregular self-maps of $\R^n$. It
is the purpose of this paper to do precisely this. Here we build on~\cite{FJ}, which is concerned with a Fatou-Julia theory for (non-uniformly) quasiregular self-maps of $S^n$. The results in~\cite{FJ} are in turn inspired by results of Sun and Yang~\cite{Sun99,Sun00,Sun01}
dealing with the case $n=2$.

We call \qr self-maps of $\R^n$ \emph{entire \qr maps}. Such a map $f$ is said to be of \emph{polynomial type} if
$\lim_{x\to \infty}|f(x)|=\infty$ and of \emph{transcendental type} if
$\lim_{x\to\infty}|f(x)|$ does not exist.
Here and in the following $|y|$ denotes the Euclidean norm of a point
$y\in\R^n$.

The iteration of entire \qr maps of polynomial type was studied in~\cite{Fletcher11}. Such maps extend to \qr self-maps of the one-point-compactifi\-cation $\R^n\cup\{\infty\}$ of $\R^n$, which we identify with the $n$-sphere $S^n$ via stereographic projection.
Hence entire \qr maps of polynomial type are also covered by the results in~\cite{FJ}. Thus we will mainly restrict to entire \qr maps of transcendental type.

The capacity of a condenser, and the distinction between sets of positive capacity and sets of zero capacity, plays an important role in the theory of \qr maps; see section~\ref{quasiregular} for the definitions.
We write $\capacity C=0$ if $C$ is a set of capacity zero
and $\capacity C>0$ otherwise.
The following definition is the same as in~\cite[Definition~1.1.]{FJ}.

\begin{definition} \label{D1a}Let $f\colon \R^n\to\R^n$ be quasiregular. Then the \emph{Julia set} $J(f)$ of $f$ is defined to be the set of all $x\in \R^n$ such that
\begin{equation} \label{1a}
\capacity\!\left(\R^n\backslash \bigcup\limits^\infty_{k=1}f^k(U)\right)=0
\end{equation}
for every neighbourhood $U$ of~$x$.
\end{definition}
It was shown in~\cite{FJ} that if $f$ is of polynomial type and the degree of $f$ is larger than the inner dilatation of~$f$, then $J(f)$ is not empty and has many properties of the classical Julia set. Also, with this hypothesis the above definition agrees with the classical one for uniformly \qr maps and thus in particular for polynomials. These results also hold in the current setting:
\begin{theorem} \label{T1a}Let $f$ be an entire \qr map of transcendental type. Then $J(f)\neq \emptyset$. In fact, $\card J(f)=\infty$.
\end{theorem}
Here $\card X$ denotes the cardinality of a set~$X$.
\begin{theorem}\label{T1b}
Let $f\colon \C\to \C$ be a transcendental entire function. Then the classical definition of $J(f)$ using non-normality agrees with the one given in Definition~\ref{D1a}.
\end{theorem}

As in the classical case, it is easily seen that $J(f)$ is completely invariant; that is, $f(x)\in J(f)$ if and only if $x\in J(f)$. We also note that if $\phi\colon \R^n\to \R^n$ is a quasiconformal homeomorphism and $g=\phi \circ f\circ \phi^{-1}$, then $J(g)=\phi(J(f))$.

Besides the escaping set $I(f)$ we also consider the set
\[BO(f)=\{x\in \R^n:(f^k(x)) \text{ is bounded}\}\]
of points with bounded orbits. For polynomials this set is called the \emph{filled Julia set} and is usually
denoted by $K(f)$, but we reserve the notation $K(f)$ for the dilatation. For polynomials and transcendental entire functions we have~\cite{Eremenko89}
$$J(f)=\partial I(f)=\partial BO(f).$$
This does not hold in the present context, as Examples~\ref{Ex7a} and~\ref{Ex7b} will show that we may have $(\partial I(f)\cap \partial BO(f))\backslash J(f)\neq \emptyset$. However, we have the following result.
\begin{theorem}\label{T1c}
Let $f$ be an entire \qr map of transcendental type. Then $J(f)\subset \partial I(f)\cap\partial BO(f)$.
\end{theorem}

One ingredient in the proof of Theorem~\ref{T1c} is the following result of independent interest.
\begin{theorem}\label{T1d} Let $f$ be an entire \qr map of transcendental type. Then $\capacity BO(f)>0$.
\end{theorem}
Theorem~\ref{T1c} follows from Theorem~\ref{T1d} together with the result of~\cite{Bergweiler09} that $I(f)$ contains continua, implying that $\capacity I(f)>0$; cf. Lemma~\ref{lemma2h} and the remark following it.

 Even for entire functions in the complex plane, $BO(f)$ need not contain continua. In this setting the Hausdorff dimension of $BO(f)$ is positive, but can be arbitrarily small~\cite{Bergweiler12}.

We say that $\xi\in \R^n$ is a \emph{periodic point} of $f\colon \R^n\to \R^n$, if there exists $p\in \N$ such that $f^p(\xi)=\xi$.
The smallest $p$ with this property is called the \emph{period} of~$\xi$.
A periodic point of period $1$ is called a \emph{fixed point}.
If a periodic point $\xi$ of period $p$ has a neighbourhood $U$ such that
$f^{pk}|_U\to \xi$  uniformly as $k\to\infty$,
then $\xi$ is called \emph{attracting} and the set of all $x\in \R^n$ satisfying $f^{pk}(x)\to \xi$ as $k\to \infty$ is called the \emph{attracting basin} of $\xi$ and denoted by $A(\xi)$. As in \cite[Theorem~1.3]{FJ} we have the following result.
\begin{theorem}\label{T1e} Let $f$ be an entire \qr map of transcendental type
with an attracting fixed point $\xi$. Then $J(f)\cap A(\xi)=\emptyset$ and $J(f)\subset \partial A(\xi)$.
\end{theorem}

Again we have equality for entire functions in the complex plane, but not in the current setting; see Example~\ref{Ex7c} below.

The \emph{forward orbit} $O^+_f(x)$ and the \emph{backward orbit} $O^-_f(x)$
of a point $x\in \R^n$ are defined by
\[O^+_f(x)=\left\{f^k(x): k \in \N\right\}\]
and
\[O^-_f(x)=\bigcup^\infty_{k=0} f^{-k}(x)=\bigcup^\infty_{k=0}\left\{y\in \R^n:f^k(y)=x\right\}.\]
For $A\subset \R^n$ we put
\[
O^\pm_f(A)=\bigcup_{x\in A} O^\pm_f(x).
\]
With this terminology~\eqref{1a} takes the form
\begin{equation} \label{1a1}
\capacity\!\left(\R^n\backslash O^+_f(U)\right)=0.
\end{equation}

The \emph{exceptional set} $E(f)$ of a \qr map $f\colon \R^n\to \R^n$ is the set of
all points with finite backward orbit.
It is a simple consequence of Picard's theorem that the exceptional set of a
non-linear entire function in the complex plane contains at most one point. Rickman~\cite{Rickman80} has extended Picard's theorem to \qr maps and shown that there exists a constant $q=q(n,K)$ such that if $f\colon \R^n\to\R^n$ is a $K$-quasiregular map of transcendental type and if $a_1,a_2,\dots, a_q\in \R^n$ are distinct, then there exists $j\in\{1,\dots,q\}$ such that $f^{-1}(a_j)$ is infinite. This implies that $E(f)$ contains at most $q-1$ points.

As in~\cite{FJ} we obtain the analogues of some further standard results of complex dynamics under the additional hypothesis of (local) Lipschitz continuity.

\begin{theorem}\label{T1f} Let $f$ be an entire \qr map of transcendental type.
Suppose that $f$ is locally Lipschitz continuous. Then
\begin{itemize}
\item[(i)\ ]
$J(f)\subset\overline{O^-_f(x)}$ for all $x\in \R^n\backslash E(f)$,
\item[(ii)\ ]
$J(f)=\overline{O^-_f(x)}$ for all $x\in J(f)\backslash E(f)$,
\item[(iii)\ ]
$\R^n\backslash O^+_f(U) \subset E(f)$ for every open set $U$ intersecting $J(f)$,
\item[(iv)\ ]
$J(f)$ is perfect,
\item[(v)\ ]
$J(f^p)=J(f)$ for all $p\in\N$.
\end{itemize}
\end{theorem}
Note that~(iii) is considerably stronger than the property~\eqref{1a} used in the definition of $J(f)$. It follows from~(iii), together with the complete invariance of $J(f)$, that
\[J(f)\backslash E(f)\subset O^+_f(U\cap J(f)) \subset J(f)\]
for every open set $U$ intersecting $J(f)$.

We also note that (v) implies that -- under the hypotheses of  Theorem~\ref{T1f} --
the conclusion of Theorem~\ref{T1e} holds not
only for attracting fixed points, but also for attracting periodic
points.

We could achieve that (v) always holds by modifying Definition~\ref{D1a} as
follows: instead of requiring that~\eqref{1a1} holds for all neighbourhoods
$U$ of $x$ we would require that
$\capacity\!\left(\R^n\backslash O^+_{f^p}(U)\right)=0$
for all neighbourhoods $U$ of $x$ and all $p\in\N$.
Suitable modifications of our arguments show that
with this definition of $J(f)$ our results about the Julia set
would also hold, and (v) would be true automatically.
However, we conjecture that Theorem~\ref{T1f} holds without the
hypothesis of local Lipschitz continuity. If this is true,
 then, in particular, (v) always holds
and so this modified definition agrees with the one given in Definition~\ref{D1a}.
Therefore we have used the
somewhat simpler definition of $J(f)$ in Definition~\ref{D1a}.

We denote the Hausdorff dimension of a subset $X$ of $\R^n$ by $\dim X$.
\begin{theorem}\label{T1g} Let $f$ be as in Theorem~\ref{T1f}. Then $\dim J(f)>0$.
\end{theorem}

In our proof of Theorem~\ref{T1a} and subsequent results we have to distinguish two cases.
It turns out that the hypothesis of Lipschitz continuity in Theorem~\ref{T1f} is needed only in one of these cases.
In order to motivate the terminology, we note that an entire function is said to have the ``pits effect'' if $|f(z)|$ is ``large'' except in ``small'' domains (which are called pits).
This concept was introduced by Littlewood and
Offord~\cite{Littlewood48}; see also~\cite{Eremenko07}. We adapt this terminology, although our definition of ``large'' and ``small'' is different from the one in the papers cited.
\begin{definition}\label{D1b}
A \qr map $f\colon \R^n\to \R^n$ of transcendental type is said to have the \emph{pits effect}
if there exists $N\in \N$ such that, for all $c>1$ and all $\varepsilon>0$, there exists $R_0$ such that if $R>R_0$, then
\[\left\{x\in \R^n:R\le |x|\le cR,\ |f(x)|\le 1\right\}\]
can be covered by $N$ balls of radius $\varepsilon R$.
\end{definition}

For example, it follows directly that if there exists a sequence $(x_k)$ tending to~$\infty$ such that $|f(x_k)|\le 1$ for all $k\in \N$ and
$\limsup_{k\to \infty}|x_{k+1}|/|x_k| <\infty$,
then $f$ does not have the pits effect. In the definition of the pits effect,
we could replace the condition that
$|f(x)|\le 1$ by $|f(x)|\le C$ for any positive constant $C$ and in fact by
$|f(x)|\le R^\alpha$ for any $\alpha>0$; see Theorem~\ref{T8a}.

\begin{theorem}\label{T1h} Let $f$ be an entire \qr map of transcendental type which does not have the pits effect. Then the conclusion of Theorem~\ref{T1f} holds.
\end{theorem}
Theorem~\ref{T1h} is a corollary of the following result.
\begin{theorem}\label{T1i}
Let $f$ be an entire \qr map of transcendental type which does not have the pits effect.
Then
$\capacity \overline{O^-_f(x)}>0$ for all $x\in \R^n\backslash E(f)$.
\end{theorem}
Together with the complete invariance of $J(f)$ we obtain the following
result from Theorem~\ref{T1i}.

\begin{cor} \label{C1a} Let $f$ be
as in Theorem~\ref{T1i}.
Then $\capacity J(f)>0$.
\end{cor}
Theorems~\ref{T1h} and \ref{T1i} apply in particular to higher dimensional analogues
of  the exponential and the trigonometric functions considered
in~\cite{Bergweiler10a,Bergweiler11,FletcherSine}; see
Examples~\ref{Ex7e} and \ref{Ex7d} below.
These functions are also locally Lipschitz continuous so that we
could apply Theorem~\ref{T1f} as well.

Another condition yielding that the conclusion of Theorem~\ref{T1i} and thus
Theorem~\ref{T1f} holds involves the \emph{branch set} $B_f$ which is defined as
the set of all points where $f$ fails to be locally injective.
The \emph{local index} $i(x,f)$ of a quasi\-regular map $f$ at
a point $x$ is defined by
$$
i(x,f) =\inf_U \sup_{y\in\R^n}\card \left( f^{-1}(y)\cap U\right),
$$
where the infimum is taken over all neighbourhoods $U$ of~$x$.
Thus
$$B_f=\{x: i(x,f)>1\}.$$
With this notation we have the following analogue of~\cite[Theorem~1.8]{FJ}.
\begin{theorem}\label{T1j}
Let $f$ be an entire \qr map of transcendental type.
If $J(f)\cap B_f=\emptyset$ or, more generally, if
the local index is bounded on $J(f)$,
then conclusions~{\rm(ii)}, {\rm(iv)} and~{\rm(v)} of Theorem~\ref{T1f} hold and $\capacity J(f)>0$.

If the local index is bounded on $\R^n$, then the conclusions of Theorems~\ref{T1f} and~\ref{T1i} hold.
\end{theorem}
We note that rational functions $f$ for which $B_f$ is contained
in attracting basins (and thus, in particular, $J(f)\cap B_f=\emptyset$)
are called \emph{hyperbolic} or \emph{expanding}. They play an important role in complex
dynamics; cf.~\cite{Milnor06,Steinmetz93}. The concept of hyperbolicity has also
been extended to transcendental dynamics, see, e.g.,~\cite{Rippon1999}.

Additional hypotheses like Lipschitz continuity or not having
the pits effect are not needed in dimension~$2$.
\begin{theorem}\label{T1k} Let $f\colon \C\to \C$ be a \qr map of
transcendental type.
Then the conclusions of Theorems~\ref{T1f} and~\ref{T1i} hold
and $\capacity J(f)>0$.
\end{theorem}
This paper is organized as follows.
In section~\ref{quasiregular} we recall the definition of
quasiregularity, capacity and some other concepts needed and we
collect a number of results that are used in the sequel.
Section~\ref{F w/o p e} deals with functions not having the pits
effect. We prove Theorems~\ref{T1a} and~\ref{T1d} for such
functions, and we also prove Theorem~\ref{T1i}.
In section~\ref{F w p e} we consider functions with pits
effect and prove Theorems~\ref{T1a} and~\ref{T1d} for such
functions.
Theorems~\ref{T1b}, \ref{T1c}, \ref{T1e}, \ref{T1f}, \ref{T1g},
\ref{T1h} and \ref{T1j} are proved in section~\ref{ProofT1bcefgh}.
The $2$-dimensional case is then considered in section~\ref{2dcase},
where Theorem~\ref{T1k} is proved. Various examples illustrating our
results are considered
in section~\ref{Examples}.
Finally, some consequences of Harnack's inequality
are discussed in section~\ref{Harnack}. In particular, we show
that the definition of the pits effect can be modified as
indicated above.

\section{Quasiregular maps, capacity and Hausdorff measure}\label{quasiregular}
We refer to the monographs~\cite{Reshetnyak89,Rickman93} for
a detailed treatment of quasi\-regular maps.
Here we only recall the definition and the main properties
needed.

For $n\geq 2$, a domain $\Omega\subset \R^n$ and
$1\leq p<\infty$, the \emph{Sobolev space}  $W^1_{p,\text{loc}}(\Omega)$
consists of the functions  $f=(f_1,\dots,f_n)\colon \Omega\to \R^n$ for which
all first order weak partial derivatives
$\partial_k f_j$ exist and are locally in $L^p$.
A~continuous map $f\in W^1_{n,\text{loc}}(\Omega)$ is called \emph{quasi\-regular}
if there exists a constant $K_O\geq 1$ such that
\begin{equation}\label{2a}
|Df(x)|^n\leq K_O J_f(x)
\quad \text{ a.e.},
\end{equation}
where $Df(x)$ denotes the derivative,
\[
|Df(x)|=\sup_{|h|=1} |Df(x)(h)|
\]
its norm, and $J_f(x)$ the Jacobian determinant.
Put
\[
\ell(Df(x))=\inf_{|h|=1}|Df(x)(h)|.
\]
The condition that~\eqref{2a} holds for some $K_O\geq 1$ is equivalent to
the condition that
\begin{equation}\label{2b}
J_f(x)\leq K_I\ell(Df(x))^n
\quad \text{ a.e.},
\end{equation}
for some $K_I\geq 1$. The smallest constants $K_O$ and $K_I$
satisfying~\eqref{2a} and~\eqref{2b} are called the \emph{outer and inner dilatation}
of $f$ and are denoted by $K_O(f)$ and $K_I(f)$.  We call
$K(f)=\max\{K_I(f),K_O(f)\}$ the (maximal) \emph{dilatation} of~$f$ and,
for $K\geq 1$, say that $f$ is $K$-\emph{quasi\-regular} if $K(f)\leq K$.

If $f$ and $g$ are quasi\-regular, with $f$ defined in the range of~$g$, then
$f\circ g$ is also quasi\-regular and~\cite[Theorem~II.6.8]{Rickman93}
\begin{equation}\label{2c}
K_I(f\circ g)\leq K_I(f)K_I(g)
\quad\text{and}\quad
K_O(f\circ g)\leq K_O(f)K_O(g)
\end{equation}
so that $K(f\circ g)\leq K(f)K(g)$.

Quasiregularity can be defined more generally for maps between
Riemannian manifolds. Here we only need the case of
quasi\-regular maps $f\colon \Omega \to \overline{\R^n}$ where $\Omega\subset\overline{\R^n}$.
Such maps are called \emph{quasi\-meromorphic}.  It turns out that
a non-constant continuous map $f\colon \Omega \to \overline{\R^n}$
is quasi\-meromorphic if $f^{-1}(\infty)$ is discrete and $f$ is quasi\-regular
in $\Omega\backslash (f^{-1}(\infty)\cup\{\infty\} )$.

Many properties  of holomorphic  functions carry over to quasi\-regular maps.
For example, non-constant quasi\-regular maps are open and discrete.
A key result already mentioned in the introduction
is Rickman's analogue~\cite{Rickman80,Rickman85} of Picard's Theorem.
\begin{lemma}\label{rickman-thm}
For $n\geq 2$ and $K\geq 1$  there
exists a constant $q=q(n,K)$ with the following property: if
$a_1,\dots,a_q\in\R^n$ are distinct, then every $K$-quasi\-regular map
$f\colon \R^n\to \R^n\backslash \{a_1,\dots,a_q\}$ is constant and every
$K$-quasi\-regular map $f\colon \R^n\to \R^n$ for which
$f^{-1}(\{a_1,\dots,a_q\})$ is finite is of polynomial type.
\end{lemma}
The number $q(n,K)$ is called the \emph{Rickman constant}.
Note that Picard's  theorem says that $q(2,1)=2$.

The following normal family analogue of Rickman's Theorem was obtained by
Miniowitz~\cite[Theorem~5]{Miniowitz82}.
Here $\chi(x,y)$ denotes the chordal distance of two points $x,y\in\RR$.
Thus  $\chi(x,y)=|\pi^{-1}(x)-\pi^{-1}(y)|$ with the
stereographic projection $\pi\colon S^n\to\RR$.
\begin{lemma}\label{lemma2b}
Let $\Omega\subset \R^n$ be a domain and
let $\mathcal{F}$ be a family of $K$-quasi\-mero\-morphic mappings
on~$\Omega$.
Let $q=q(n,K)$ be the Rickman constant and
suppose that there exists $\delta>0$ with the following property: for each
$f\in\mathcal{F}$ there exist $a_1^f,\dots,a_{q+1}^f\in\RR$
satisfying $\chi(a_i^f,a_j^f)\geq\delta$ for $i\neq j$
such that $f(x)\neq a_j^f$ for all $x\in \Omega$ and $1\leq j\leq q+1$.
Then $\mathcal{F}$ is normal.
\end{lemma}
A family $\mathcal{F}$ of functions quasiregular (or quasi\-mero\-morphic) in a domain
$\Omega$ is called \emph{quasinormal} if
every sequence in $\mathcal{F}$
has a subsequence which converges locally uniformly in
$\Omega\backslash E$ for some finite subset $E$ of $\Omega$.
Here the subset $E$ may depend on the sequence.
The following simple consequence of the maximum principle
is useful in dealing with sequences converging outside a finite set;
see, e.g., \cite[Lemma~3.2]{Bergweiler06}.
\begin{lemma} \label{lemma2g}
Let $\Omega\subset \R^n$ be a domain and let $(f_k)$ be
a non-normal sequence of functions
which are $K$-quasi\-regular in~$\Omega$.
If $(f_k)$ converges locally
uniformly in $\Omega\backslash E$ for some finite set~$E$,
then $f_k\to\infty$ in $\Omega\backslash E$.
\end{lemma}

The next result is the analogue of Liouville's Theorem for \qr maps;
see, e.g., \cite[Lemma~3.4]{Bergweiler06}.
Here  $M(r,f)$ denotes the maximum modulus of a \qr map~$f$; that is,
$M(r,f)=\max_{|x|=r}|f(x)|$.
\begin{lemma}\label{lemma2a}
Let $f$ be an entire \qr map of transcendental type.
Then
$$
\lim_{r\to\infty}\frac{\log M(r,f)}{\log r}=\infty .
$$
\end{lemma}

An important tool in the theory of \qr mappings is
the capacity of a condenser. We recall this concept briefly.
For an open set $A\subset\R^n$ and a non-empty compact subset $C$ of $A$, the pair
$(A,C)$ is called a \emph{condenser} and its  \emph{capacity}
$\capacity (A,C)$ is defined by
\[
\capacity (A,C)=\inf_u\int_A\left|\nabla u\right|^n dm,
\]
where the infimum is taken over all non-negative functions
$u\in C^\infty_0(A)$ satisfying $u(x)\geq 1$ for all $x\in C$.
Equivalently, one may take the infimum over all non-negative
$u\in W^1_{n,\text{loc}}(A)$ with compact support
and $u(x)\geq 1$ for all $x\in C$.
It follows directly from the definition that
\begin{equation}\label{lemma2j}
\capacity(A',C')\leq \capacity(A,C)
\quad\text{if}\
A'\supset A\ \text{and}\ C'\subset C.
\end{equation}
If $\capacity (A,C)=0$ for some bounded open set $A$ containing~$C$,
then  $\capacity (A',C)=0$ for every bounded open set $A'$ containing~$C$;
see~\cite[Lemma III.2.2]{Rickman93}.
In this case
we say that $C$ is of  \emph{capacity zero}. Otherwise we say that $C$ has
\emph{positive capacity}. We denote this by
$\capacity C=0$ or $\capacity C>0$, respectively.

For an unbounded closed subset $C$ of $\R^n$ we say that $C$ has capacity zero
if every compact subset of $C$ has capacity zero. Equivalently, we can
define condensers in $\RR$ and consider $\capacity (A,C\cup\{\infty\})$
for open subsets $A$ of $\RR$ with $\overline{A}\neq \RR$ which
contain $C\cup\{\infty\}$.

For $a\in\R^n$ and $r>0$, let $B^n(a,r)=\{x\in\R^n:|x-a|<r\}$ be the open ball,
$\overline{B^n}(a,r)$ the closed ball and $S^{n-1}(a,r)=\partial
B^n(a,r)$ the sphere of radius $r$ centred at~$a$.
We write $B^n(r)$, $\overline{B^n}(r)$ and $S^{n}(r)$
instead of $B^n(0,r)$, $\overline{B^n}(0,r)$ and $S^{n}(0,r)$.
If there is no risk of confusion, we omit the superscript $n$ for
the dimension.

For a quasi\-regular map $f\colon \Omega \to \RR$, a point $y\in\RR$ and
a Borel set $E$ such that $\overline{E}$ is a compact
subset of $\Omega$, we
denote by $n(E,y,f)$ the number of $y$-points of $f$ in~$E$,
counted according to multiplicity. Thus
$$
n(E,y,f)=\sum_{x\in f^{-1}(y)\cap E} i(x,f).
$$
The average of $n(E,y,f)$ over all $y\in \RR$ is denoted by $A(E,f)$.
Thus~\cite[p.~80]{Rickman93}
$$
A(E,f)=\frac{1}{\omega_{n}}\int_{\R^n} \frac{n(E,y,f)}{(1+|y|^2)^n}dy
=\frac{1}{\omega_{n}}\int_{E}
\frac{J_f(x)}{(1+|f(x)|^2)^n}dx,
$$
where $\omega_n$ denotes the volume of the $n$-dimensional unit ball.

Note that Rickman~\cite{Rickman93} identifies $\RR$ with $S^n(e_{n+1}/2,1/2)$
while we  have used $S^n=S^n(0,1)$, implying that the formulas differ by
a factor~$2^n$. Here and in the following $e_k$ denotes the $k$-th unit vector.

We will write $n(r,y,f)$ and  $A(r,f)$
instead of $n\!\left(\overline{B}(r),y,f\right)$
and $A\!\left(\overline{B}(r),f\right)$.

The following result~\cite[Theorem~II.10.11]{Rickman93} gives a
connection between capacity and quasiregularity.
\begin{lemma}\label{lemma2i}
Let $f\colon \Omega\to\R^n$ be quasi\-regular, let
$(A,C)$ be a condenser in $\Omega$ and put $m=\inf_{y\in f(C)}n(C,y,f)$. Then
$$
\capacity (f(A),f(C))\leq\frac{K_I(f)}{m} \capacity (A,C).
$$
\end{lemma}

The following result was proved in~\cite[Theorem 3.2]{FJ}, based on ideas
from~\cite{Mattila79}.
\begin{lemma}\label{lemma2c}
Let $F\subset \RR$ be a set of positive capacity and let
$\theta>1$. Then there exists a constant $C$ depending only
on~$n$, $F$ and $\theta$ such that if $f\colon B^n(\theta r)\to \RR\backslash F$
is quasi\-regular, then
$A(r,f)\leq C\, K_I(f)$.
\end{lemma}

The condenser
$$
E_{G}(t)=\left(B^n(1),[0,t e_1]\right)
$$
is called the \emph{Gr\"otzsch condenser}. It has the following important
extremal property~\cite[Lemma~III.1.9]{Rickman93}.
\begin{lemma}\label{lemma2k}
Let $(A,C)$ be a condenser with $A\subset B^n(r)$.
Suppose that $C$ is connected and that $0,x\in C$ where $x\in B^n(r)\backslash\{0\}$.
Then
$$
\capacity (A,C)\geq \capacity E_G(|x|/r).
$$
\end{lemma}
We shall also need the following bound for the capacity of the Gr\"otzsch
condenser~\cite[Lemma~III.1.2]{Rickman93}.
\begin{lemma}\label{lemma2l}
There exists a constant $\lambda_n$ depending only on $n$ such that
$$
\capacity E_G(t)\geq \omega_{n-1}\left(\log\frac{\lambda_n}{t}\right)^{1-n}
$$
for $0<t<1$.
\end{lemma}

We denote the (Euclidean) diameter of a subset $A$ of $\R^n$ by $\diam A$.
For $\eta>0$,
an increasing, continuous function $h\colon (0,\eta)\to (0,\infty)$
satisfying $\lim_{t\to 0}h(t)=0$ is called a \emph{gauge function}.
For $A\subset\R^n$ and $\delta>0$, we call a sequence $(A_j)$ of subsets of $\R^n$ a
$\delta$-\emph{cover} of $A$ if $\diam A_j <\delta$ for all $j\in\N$ and
$$
A\subset\bigcup_{j=1}^{\infty}A_j.
$$
We put
$$
H_h^{\delta}(A)=
\inf\left\{\sum_{j=1}^{\infty}h(\diam A_j):
(A_j)\text{ is }\delta\text{-cover of }A\right\}
$$
and call
$$
H_h(A)=\lim_{\delta\to 0}H_h^{\delta}(A)
$$
the \emph{Hausdorff
measure} of $A$ with respect to the gauge function~$h$.
Note that since $H_h^\delta(A)$ is a non-increasing function of $\delta$,
the limit defining $H_h(A)$ exists.
The possibility that $H_h^\delta(A)=\infty$ is allowed,
meaning that $\sum_{j=1}^\infty h(\diam A_j)$ diverges for all $\delta$-covers $(A_j)$.

In the special case that $h(t)=t^s$ for some $s>0$, we call $H_h(A)$ the
$s$-\emph{dimensional Hausdorff measure}. There exists $d\in [0,n]$ such that
$H_{t^s}(A)=\infty$ for $0<s<d$ and $H_{t^s}(A)=0$ for $s>d$. This value $d$ is
called the \emph{Hausdorff dimension} of $A$ and is denoted by $\dim A$.

Recall that a function $f\colon X\to\R^n$ where $X\subset \R^n$ is said to
be \emph{H\"older continuous} with exponent $\alpha$
if there exists $L>0$ such that
$$
|f(x)-f(y)|\leq L|x-y|^\alpha\quad \text{for all}\  x,y\in X.
$$
In the special case that $\alpha=1$ we say that $f$ is
\emph{Lipschitz continuous} with \emph{Lipschitz constant}~$L$.
The following result was proved in~\cite[Corollaries 9.1 and 9.2]{FJ}.
\begin{lemma}\label{lemma2d}
Let $X\subset \R^n$ be compact, $f\colon X\to \R^n$ and $\delta>0$.
Suppose that each $y\in X$ has $m$ pre-images $x_1,\dots,x_m$ satisfying
$|x_i-x_j|\geq\delta$ for $i\neq j$. Let $x\in X$.
\begin{itemize}
\item[(i)]
If $f$ satisfies a Lipschitz condition with Lipschitz constant~$L>1$, then
$$
\dim \overline{O^-_f(x)}\geq \frac{\log m}{\log L}.
$$
\item[(ii)]
If $f$ satisfies a H\"older condition with exponent $\alpha<1$, then
$$
H_h\left(\overline{O^-_f(x)}\right)>0
$$
for
$$
h(t)
=
\left(\log\frac{1}{t}\right)^{(\log m)/(\log \alpha)}.
$$
\end{itemize}
\end{lemma}
In~\cite{FJ} it is actually assumed that $f$ maps $X$ to $X$
and the conclusion concerns $\dim X$ and $H_h(X)$. However, the proof yields the
above result.

Part (ii) of the above result can be applied in particular to
\qr maps by the following result~\cite[Theorem III.1.11]{Rickman93}.
\begin{lemma}\label{lemma2m}
Let $f\colon \Omega\to\R^n$ be quasiregular. Then
$f$ is locally H\"older continuous with exponent
$K_I(f)^{1/(1-n)}$.
\end{lemma}

The following result connecting capacity and Hausdorff measure
can be found in~\cite{Wallin77}.
\begin{lemma}\label{lemma2f}
Let $X\subset\R^n$ be compact and $\eps>0$.  If
$H_h(X)>0$
for
$$
h(t) =
\left(\log\frac{1}{t}\right)^{1-n-\eps},
$$
then $\capacity X>0$.
\end{lemma}
An immediate consequence is the following result.
\begin{lemma}\label{lemma2h}
Let $X\subset\R^n$ be compact.
If $\dim X>0$, then $\capacity X>0$.
\end{lemma}
In particular, it follows that $\capacity X>0$ if $X$ contains
a non-degenerate continuum.

\section{Functions without pits effect}\label{F w/o p e}
 In this section, we prove Theorem~\ref{T1a} and Theorem~\ref{T1d} for functions which do not have the pits effect.
We also prove Theorem~\ref{T1i} which deals only with such functions.

 Throughout this section, let $f\colon \R^n\to \R^n$ be an entire \qr map of transcendental type which does not have the pits effect. We fix a large positive integer $N$ and obtain $c>1$, $\varepsilon >0$ and a sequence $(R_m)$ tending to $\infty$ such that $\{x\in \R^n:\ R_m\le|x|\le cR_m,\ |f(x)|\le 1\}$ cannot be covered by $N$ balls of radius $\varepsilon R_m$. We consider the functions $h_m\colon \R^n\to \R^n$,
 \begin{equation}\label{3a} h_m(x)=f(R_mx).
 \end{equation}
 With $P=\{x\in \R^n: 1\le|x|\le c\}$ we find that
 \[\{x\in P:|h_m(x)|\le 1\}\]
 cannot be covered by $N$ balls of radius $\varepsilon$. Thus there exist $x_1^m,\dots,x_N^m\in P$ satisfying $|x_i^m-x_j^m|\ge \varepsilon$ for $i\neq j$ such that $|h_m(x_j^m)|\le 1$ for $j\in\{1,\dots, N\}$. Passing to a subsequence if necessary, we may assume that the sequences $(x_j^m)_{m\in \N}$ converge for $j\in\{1,\dots,N\}$, say $x_j^m\to x_j$ as $m\to \infty$. Then $|x_i-x_j|\ge \varepsilon$ for $i\neq j$. We fix $r_1,\dots, r_N$ satisfying $1\le r_1<r_2<\dots <r_N\le c$ and choose $y^m_j$ such that $|y_j^m|=r_j$ and
 \[|h_m(y_j^m)|=M(r_j,h_m)=M(R_mr_j,f).\]
Again we may assume that the sequences $(y_j^m)_{m\in \N}$ converge, say $y_j^m\to y_j$ as $m\to \infty$. We may choose pairwise disjoint curves $\gamma_j$ which connect $x_j$ with $y_j$
and small neighbourhoods $U_j$ of the curves $\gamma_j$
such that their closures $\overline{U}_j$ are pairwise disjoint.
Then $x_j\in U_j$ and $y_j\in U_j$ for $j\in\{1,\dots,N\}$, and there exists $\delta>0$ with
\[\dist(U_i,U_j)=\inf\limits_{v\in U_i,w\in U_j }|v-w|\ge \delta\]
for $i\neq j$. Since $|h_m(x_j^m)|\le 1$ and $x_j^m\to x_j$ while $|h_m(y_j^m)|\to \infty$ and $y_j^m\to y_j$, we find that the sequence $(h_m)$ is not normal in any of the domains $U_j$. In fact, no subsequence of $(h_m)$ is normal.

We shall also consider the functions $g_m\colon \R^n\to \R^n$,
 \begin{equation}\label{3b} g_m(x)=\frac{f(R_mx)}{R_m}=\frac{h_m(x)}{R_m}.
 \end{equation}
Using Lemma~\ref{lemma2a} we find that $|g_m(y_j^m)|\to \infty$ as $m\to \infty$.
Since $|g_m(x_j^m)|\le 1/R_m$ this implies that no subsequence of $(g_m)$ is normal in any of the domains $U_j$.
\begin{proof}[Proof of Theorem \ref{T1a} for functions without pits effect]
With $g_m$ and $U_j$ as defined above we deduce from Lemma~\ref{lemma2b} that if $m\in\N$ is large enough, say $m\ge M$, and if $j\in\{1,\dots,N\}$, then $g_m(U_j)\supset U_i$ for at least $N-q$ values of $i\in \{1,\dots,N\}$.
This implies that if $k\in \N$, $j\in\{1,\dots,N\}$ and $m\ge M$, then,
counting multiplicities, $g_m^k(U_j)$ covers at least $(N-q)^k$ of the domains $U_i$.
This means that with $L=(N-q)^k$ there exist pairwise disjoint subsets $V_1,\dots,V_L$ of $U_j$ such that if $\ell\in \{1,\dots,L\}$, then $g_m^k(V_\ell)=U_i$ for some $i\in \{1,\dots,N\}$. Hence, for each $j\in \{1,\dots,N\}$, there exists $i\in\{1,\dots,N\}$ such that
\[n(U_j,y,g_m^k)\ge\frac{(N-q)^k}{N}\quad \text{for all}\ y \in U_i.\]
 This implies that
 \begin{equation}\label{3c} A(U_j,g_m^k)\ge C_1\frac{(N-q)^k}{N}
 \end{equation}
for some $C_1>0$ and all $j\in\{1,\dots,N\}$.

Suppose now that $J(g_m)\cap \overline{U}_j=\emptyset$ for some $j\in \{1,\dots,N\}$. Then for each $x\in \overline{U}_j$ there exists $r_x>0$ such that
$B(x,2r_x)\cap J(g_m)=\emptyset$. Hence
\[A(B(x,r_x),g_m^k)\le c_xK_I(g_m^k)\le c_xK_I(g_m)^k\]
for some $c_x>0$ by Lemma~\ref{lemma2c}, \eqref{2c} and the definition of $J(g_m)$. Since $\overline{U}_j$ can be covered by finitely many balls $B(x,r_x)$, we obtain
\begin{equation}\label{3d} A(U_j,g_m^k)\le C_2K_I(g_m)^k=C_2K_I(f)^k.
 \end{equation}
Choosing $N>K_I(f)+q$ we obtain a contradiction from~\eqref{3c} and~\eqref{3d}.
Thus $J(g_m)\cap\overline{U}_j\neq \emptyset$ for $j\in \{1,\dots,N\}$.
Since $f$ is conjugate to $g_m$ by the linear map $x\mapsto R_m x$, we deduce that $\card J(f)\ge N$. We obtain $\card J(f)=\infty$.
\end{proof}
\begin{proof}[Proof of Theorem \ref{T1i}]
Let $g_m$ and $U_j$ be as before. Let $p$ be the cardinality of the set of all $i\in \{1,\dots,N\}$ such that $g_m(U_j)\supset U_i$ for at most $N/2$ values of $j\in \{1,\dots,N\}$. Since each $U_j$ has $N-q$ subsets mapped onto distinct domains $U_i$, we find that
\[p\frac{N}{2}+(N-p)N\ge N(N-q),\]which is equivalent to $p\le 2q$. We choose $N$ divisible by 4 such that $N\ge 8q$. Then $N/4\ge p$.
Hence $3N/4$ of the domains $U_i$ are contained in $N/2$ of the domains $g_m(U_j)$. With $L=3N/4$ we may assume that $U_1,\dots,U_L$ are contained in $N/2$ of the domains $g_m(U_j)$;
that is, for each $i\in \{1,\dots,L\}$,
\[\card\left\{j\in\{1,\dots,N\}:g_m(U_j)\supset U_i\right\}\ge \frac{N}{2}.\]
Hence
$\card\{j\in\{1,\dots,L\}:g_m(U_j)\supset U_i\}\ge N/4$, which implies that
\[\card\left\{j\in\{1,\dots,L\}:g_m\!\left(\overline{U}_j\right)\supset
\overline{U}_i\right\}\ge \frac{N}{4},\]
for each $i\in \{1,\dots,L\}$.

It now follows from Lemma~\ref{lemma2d}, (ii), applied to $X=\bigcup_{j=1}^L\overline{U}_j$,
and  Lemma~\ref{lemma2m} that if $y\in X$, then
\[H_h\left(\overline{O^-_{g_m}(y)}\right)>0,\]
where
\[h(t)=\left(\log\frac{1}{t}\right)^{-\beta}
\quad\text{with}\
\beta=(n-1)\frac{\log(N/4)}{\log K_I(f)}.\]
Choosing $N>4K_I(f)=4K_I(g_m)$ and using Lemma~\ref{lemma2f} we obtain
\begin{equation}
\label{3d1}\capacity\overline{O^-_{g_m}(y)}>0
\quad\text{for}\
y\in X.
\end{equation}
Putting
\[V_j^m =R_m U_j =\{R_m x:x\in U_j\}\]
we obtain
\begin{equation}
\label{3e}\capacity\overline{O^-_{f}(v)}>0
\quad\text{for}\
v\in \bigcup_{j=1}^L\overline{V}_j^m,
\end{equation}
provided $m\geq M$.

Let now $x\in \R^n\backslash E(f)$.
Then $\card O^-_f(x)=\infty$ by the definition of $E(f)$.
Recall that no subsequence of the sequence $(h_m)$ defined
by~\eqref{3a} is normal in any of the domains $U_j$.
Lemma~\ref{lemma2b} implies that if $j\in\{1,\dots,N\}$
and if $m$ is sufficiently large, then there exists
$u\in U_j$ such that $h_m(u)\in O^-_f(x)$.
Putting $v=R_m u$ we obtain $v\in V_j^m$ and $f(v)\in O^-_f(x)$.
This implies that $v\in O^-_f(x)$ and hence that $O^-_f(v)\subset O^-_f(x)$.
Now~\eqref{3e} yields $\capacity \overline{O^-_f(x)}>0$.
\end{proof}
\begin{proof}[Proof of Theorem \ref{T1d} for functions without pits effect]
Using the terminology of the previous proof, and putting
\[U=\bigcup_{j=1}^L U_j
\quad\text{and}\quad
V=\bigcup_{j=1}^L V_j^m,\]
as well as $G=g_m|_U$ and $F=f|_V$, Lemma~\ref{lemma2d} and Lemma~\ref{lemma2f}
actually show that~\eqref{3d1} and~\eqref{3e} can be replaced by
\begin{equation}
\label{3f}\capacity\overline{O^-_G(y)}>0
\quad\text{for}\
y\in U
\end{equation}
and
\begin{equation}
\label{3g}\capacity\overline{O^-_F(v)}>0
\quad\text{for}\
v\in V.
\end{equation}
By a result of Siebert~\cite{Siebert06}, $f$ has infinitely many periodic
points of period $p$ for all $p\geq 2$. In particular, $f$ has a
periodic point $\xi$ of period $2$ with $\xi\notin E(f)$.
As before, there exists $v\in O^-_f(\xi)\cap V$, provided $m$ is
large enough. Let $l$ be such that $f^l(v)=\xi$.
Then
\[W=O^-_F(v)\cup \left\{f^k(v):0\leq k\leq l+1\right\}\]
satisfies $f(W)\subset W$ and hence $f(\overline{W})\subset \overline{W}$.
Thus $\overline{W}\subset BO(f)$. Now the conclusion follows from~\eqref{3g}.
\end{proof}

\section{Functions with pits effect}\label{F w p e}
In this section we prove Theorem~\ref{T1a} and Theorem~\ref{T1d} for entire \qr maps of transcendental type which have the pits effect. Let $f\colon \R^n\to\R^n$ be such a map. Then there exists a sequence $(x_m)$ tending to $\infty$ such that $|f(x_m)|\le 1$ while $|f(x)|>1$ for $|x_m|<|x|<3|x_m|$. We put $R_m=|x_m|$ and define $g_m$ by~\eqref{3b}; that is, $g_m(x)=f(R_mx)/R_m$.

Using Harnack's inequality one can show that $(g_m)$ is quasinormal; see
Remark~\ref{R8a}.
However, the quasinormality of $(g_m)$ is not essential and thus we briefly indicate how the argument can be completed if we assume that $(g_m)$ is not quasinormal. Given $N\in \N$, we may then assume, passing to a subsequence if necessary, that there exist distinct points $z_1,\dots z_N\in \R^n$ such that no subsequence of $(g_m)$ is normal at any of these points. Choose neighbourhoods $U_1,\dots,U_N$ of these points with pairwise disjoint closures. As in
section~\ref{F w/o p e}
we now see that if $N>K_I(f)+q$, then $J(g_m)\cap \overline{U}_j\neq \emptyset$
for all $j\in \{1,\dots,N\}$, provided $m$ is large enough.
Hence $J(f)\neq \emptyset$ and in fact $\card J(f)=\infty$,
which proves Theorem~\ref{T1a} in this case.

Moreover, putting again $L=3N/4$ and proceeding as in the proof of Theorem~\ref{T1i} we see that~\eqref{3d1} holds if $N>4K_I(f)$. The arguments used in~\cite{Siebert06} and~\cite{Bergweiler06} show that $\bigcup^L_{j=1}U_j$ contains a periodic point
$\xi$ of $g_m$. As in the proof of Theorem~\ref{T1d} for functions without pits
effect we deduce from $\eqref{3f}$ that $\capacity BO(g_m)>0$. Hence $\capacity BO(f)>0$, which proves Theorem~\ref{T1d} in this case.

We will thus assume from now on that $(g_m)$ is quasinormal. Passing to a subsequence if necessary, we may assume that $(g_m)$ converges in $\R^n\backslash F$ where $F$ is a finite set. We may assume that no subsequence of $(g_m)$ is normal at any point of $F$ since this can be achieved by deleting points from $F$ and passing to a subsequence of $(g_m)$.
Using Lemma~\ref{lemma2a} it is easy to see that no subsequence of $(g_m)$ is normal at~$0$.
Thus Lemma~\ref{lemma2g} implies that $g_m\to \infty$ locally uniformly in $\R^n\backslash F$. Since $|f(x_m)|\le 1$ we conclude that $F$ contains a point of norm~$1$. Clearly we also have $0\in F$. Thus $N=\card F\ge 2$. Let $F=\{y_1,\dots,y_N\}$ and choose
$$M>\max\limits_{k=1,\dots,N}|y_k|$$
 and $r>0$ such that the closed balls of radius $r$ around the points $y_k$ are pairwise disjoint and contained in $B(0,M)$.

For large $m$ we have
\[\inf_{x\in B(y_k,r)}|g_m(x)|<M\quad\text{for}\  1\le k\le N\]
while
\[|g_m(x)|>M\quad\text{for}\  x\in \overline{B}(0,M)\backslash \bigcup^N_{k=1}B(y_k,r).\]
Denote by $\overline{U}_{m,1},\dots,\overline{U}_{m,\ell_m}$ the components of
$g_m^{-1}(\overline{B}(0,M))\cap B(0,M)$. Then each $B(y_k,r)$ contains at least
one $\overline{U}_{m,j}$ so that $\ell_m\ge N\ge 2$. Denote by $U_{m,j}$ the interior of $\overline{U}_{m,j}$.
(We do not assume here that $U_{m,j}$ is connected.) Then $g_m\colon U_{m,j}\to B(0,M)$ is a proper map. Since
\[
\left|g_m\left(\frac{x_m}{R_m}\right)\right|=\frac{|f(x_m)|}{R_m}\le \frac{1}{R_m}<M,
\]
 there exists $j$ such that $x_m/R_m\in U_{m,j}$.
As $g_m\colon U_{m,j}\to B(0,M)$ is proper, $U_{m,j}$ contains a zero $\zeta_m$ of $g_m$. Putting $z_m=R_m\zeta_m$ we thus have $f(z_m)=0$ and $|z_m-x_m|<2rR_m=2r|x_m|$. Since $r$ can be chosen arbitrarily small, we deduce that $f$ has infinitely many zeros.

Denote by $d_{m,j}$ the degree of the proper map $g_m\colon U_{m,j}\to B(0,M)$ and put
\[d_m=\sum^{\ell_m}_{j=1}\; d_{m,j}.\]
Then $d_{m,j}$ equals the number of zeros of $g_m$ in $U_{m,j}$ and thus
\[d_m=n(M,0,g_m)=n(M R_m,0,f)\to \infty\]
as $m\to \infty$. In particular we have $d_m>K_I(f)^2$ for large~$m$. This condition will be needed in the proof of Theorem~\ref{T1d}. For Theorem~\ref{T1a} the weaker condition $d_m>K_I(f)$ would suffice.

From now on we drop the index $m$ and put
\[
U=\bigcup^{\ell_m}_{j=1} U_{m,j}, \
G=g_m, \
V_j=U_{m,j}, \
L=\ell_m, \
D_j=d_{m,j} \ \text{and} \
D=d_m.
\]

Then  $L\geq 2$ and
$V_1,\dots,V_L$ are open subsets of $B(0,M)$ with disjoint closures such that $G\colon V_j\to B(0,M)$ is a proper map of degree $D_j$ for $1\le j \le L$ and
\begin{equation}\label{4a} D=\sum^L_{j=1}D_j>K_I(G)^2.
\end{equation}
Note here that $K_I(G)=K_I(f)$.

For $i,j\in \{1,\dots,L\}$ let $\overline{W}_1^{i,j},\dots,\overline{W}_t^{i,j}$ be the components of $G^{-1}(\overline{V}_j)\cap V_i$, with $t=t_{i,j}\ge 1$.
Denote by $W_s^{i,j}$ the interior of $\overline{W}_s^{i,j}$. Then $G\colon W_s^{i,j}\to V_j$ is a proper map and
\begin{equation}\label{4a0}
\sum^L_{i=1}\sum^{t_{i,j}}_{s=1} \deg\!\left(G\colon W_s^{i,j}\to V_j\right)=D.
\end{equation}
Now $G^2\colon W_s^{i,j}\to B(0,M)$ is also a proper map and
\begin{equation*}
\sum^L_{i=1}\sum^{t_{i,j}}_{s=1} \deg\!\left(G^2\colon W_s^{i,j}\to B(0,M)\right)=D_jD.
\end{equation*}
Putting
\[L_2=\sum^L_{i=1}\sum^L_{j=1}t_{i,j}\]
and writing
\[\left\{W_s^{i,j}: 1\le i,j\le L,\ 1\le s\le t_{i,j}\right\}=\left\{V^2_\ell:1\le \ell\le L_2\right\}\]
we obtain $L_2$ open subsets $V_\ell^2$ of $B(0,M)$ with pairwise disjoint closures such
that $G^2\colon V_\ell^2\to B(0,M)$ is a proper map with
\[
\sum_{\ell=1}^{L_2} \deg \!\left(G^2\colon V_\ell^2\to B(0,M)\right)=D^2.
\]
Note that $L_2\ge L^2$.

Inductively we find that for $k\in \N$ there exist $L_k\geq L^k$
and open subsets $V_1^k,\dots,V_{L_k}^k$ of $B(0,M)$ with pairwise
disjoint closures $\overline{V}_1^k,\dots,\overline{V}_{L_k}^k$ such that
$G^k\colon V_\ell^k\to B(0,M)$ is a proper map for $1\le \ell \le L_k$, with
\begin{equation}\label{4b} \sum^{L_k}_{\ell=1}\deg\!\left(G^k\colon V^k_\ell\to B(0,M)\right)=D^k.
\end{equation}
By construction,
\[\bigcup^{L_{k+1}}_{\ell=1}\overline{V}^{k+1}_{\ell}
\subset \bigcup^{L_{k}}_{\ell=1} V^k_\ell.\]

\begin{proof}[Proof of Theorem \ref{T1a} for functions with pits effect]
Let $V_1^k,\dots,V_{L_k}^k$ be as above and put
\[V^k=\bigcup^{L_k}_{\ell=1}V^k_\ell.\]
Then $n(V^k,a,G^k)= D^k$ for all $a\in B(0,M)$ by~\eqref{4b}. Thus
\begin{equation}\label{4c} A(V^k,G^k)\ge C_1D^k
\end{equation}
for some $C_1>0$. Suppose now that $\overline{V}^k\cap J(G)=\emptyset$. Then for each $x\in \overline{V}^k$ there exists $r_x>0$ such that
$B(x,2r_x)\cap J(G)=\emptyset$.
As in section~\ref{F w/o p e}, we deduce
from Lemma~\ref{lemma2c} and the definition of $J(G)$ that
\[A(B(x,r_x),G^k)\le c_xK_I(G^k)\le c_xK_I(G)^k\]
for some $c_x>0$.
Covering $\overline{V}^k$ by finitely many balls $B(x,r_x)$ we obtain
\[A(V^k,G^k)\le C_2K_I(G)^k\]
for some $C_2>0$. Since $K_I(G)=K_I(f)$ this contradicts~\eqref{4a} and~\eqref{4c} for
large~$k$. Thus $\overline{V}^k\cap J(G)\neq \emptyset$.

Since $G^k\colon V_\ell^k\to B(0,M)$ is proper and thus in particular surjective, $\overline{V}^k\subset B(0,M)$ and $J(G)$ is completely invariant, it follows that every set $V^k_\ell$ meets $J(G)$.
Thus $\card  J(G)\ge L_k$ for all $k$ and hence $\card  J(G)=\infty$.
As $G$ is conjugate to $f$,
the conclusion follows.
\end{proof}

\begin{proof}[Proof of Theorem \ref{T1d} for functions with pits effect]
We again use the terminology introduced above. Suppose first that
\begin{equation}\label{4d} L_k>K_I(G^k)\quad  \text{for some}\ k\in \N.
\end{equation}
By \cite[Lemma~2.1.5]{Siebert06}, each $V_j$ contains a fixed point $v_j$ of~$G$.
Similarly as in section~\ref{F w/o p e}
 we can now deduce from~\eqref{4d}, together with
Lemma~\ref{lemma2d} and Lemma~\ref{lemma2f},
that $\capacity\overline{O^-_{G^k}(v_j)}>0$ and hence that $\capacity BO(G)>0$.
Thus $\capacity BO(f)>0$.

Suppose now that~\eqref{4d} does not hold. Then
\begin{equation}\label{4e} L_k\le K_I(G^k)\le K_I(G)^k
\end{equation}
for all $k\in \N$. Put
\[X=\bigcap^\infty_{k=1}\bigcup^{L_k}_{\ell=1}\overline{V}_\ell^k.\]
Clearly $X$ is compact and $X\subset BO(G)$. Thus it suffices to prove that $\capacity X>0$. Suppose that this is not the case.
Then $\dim X=0$ by Lemma~\ref{lemma2h}.

This implies that for each $\eta>0$ there exists $k_0$ such that
\begin{equation}\label{4f} \diam V^k_\ell <\eta \quad \text{for}\ k\ge k_0 \
\text{and}\ 1\le  \ell\le L_k.
\end{equation}
In fact, suppose that there exist $\eta>0$ and $x_j,y_j\in V^{k_j}_{\ell_j}$ with $k_j\to \infty$ such that $|x_j-y_j|\ge \eta$. We may assume that $(x_j)$ and $(y_j)$ converge, say $x_j\to x_0$ and $y_j\to y_0$. Then $x_0,y_0\in X$. Since $\dim X<1$, there exists a hyperplane $H$ in $\R^n\backslash X$ that separates $x_0$ and $y_0$. For large $j$ the points $x_j$ and $y_j$ are on opposite sides of the hyperplane $H$ and thus there exists $z_j\in \overline{V}_j^{k_j}\cap H$. We may assume that $(z_j)$ converges, say $z_j\to z_0$, since otherwise we may pass to a subsequence. Then $z_0\in X\cap H$ since $X\cap H$ is compact. This contradicts $H\subset \R^n\backslash X$. Thus~\eqref{4f} holds.

It follows from~\eqref{4b} and~\eqref{4e} that for all $k\in \N$ there exists $\ell_k\in \{1,\dots,L_k\}$ such that
\begin{equation}\label{4g}
\deg\!\left(G^k\colon V_{\ell_k}^k\to B(0,M)\right)\ge \frac{D^k}{L_k}\ge \left(\frac{D}{K_I(G)}\right)^k.
\end{equation}
We fix a large $k$ and put $W_k=V_{\ell_k}^k$ and $W_j=G^{k-j}(W_k)$ for $0\le j\le k-1$ so that $W_0=B(0,M)$. Then $G\colon W_j\to W_{j-1}$ is a proper map for $1\le j\le k$. Put
\[P_j=\deg\!\left(G\colon W_j\to W_{j-1}\right).\]
Then
\[\prod^k_{j=1}P_j=\deg\!\left(G^k\colon W_k\to W_0\right)\]
so that
\begin{equation}\label{4h} \prod^k_{j=1}P_j \ge  \left(\frac{D}{K_I(G)}\right)^k
\end{equation}
by~\eqref{4g}. Let $w_k\in G^{-k}(0)\cap W_k$ and put $w_j=G^{k-j}(w_k)$ for $0\le j\le k-1$. Thus $w_j\in W_j$ for $0\le j\le k$ and $w_0=0$.

Let now $\varepsilon>0$ and choose $\delta >0$ such that $|G(x)-G(y)|<\varepsilon$ for $x,y\in \bigcup^L_{j=1}V_j$ satisfying $|x-y|<\delta$. We take $\eta<\min\{\varepsilon,\delta\}$ and choose $k_0$ such that~\eqref{4f} is satisfied.

Assuming that $k>k_0$, we denote by $Y_j$, for $k_0<j\le k$, the component
of $G^{-1}(B(w_{j-1},\varepsilon))$ that contains $w_j$. Then $Y_j\supset B(w_j,\delta)\supset\overline{W}_j$.

By Lemma~\ref{lemma2i} we have
\begin{align*}\capacity\!\left(Y_j,\overline{W}_j\right)
& \ge \frac{P_j}{K_I(G)}\capacity\! \left(G(Y_j),G(\overline{W}_j)\right)\\
& =\frac{P_j}{K_I(G)}\capacity\!\left(B(w_{j-1},\varepsilon),\overline{W}_{j-1}\right).
\end{align*}

By the extremality of the Gr\"otzsch condenser (Lemma~\ref{lemma2k}) and the lower bound for its capacity given by Lemma~\ref{lemma2l} we have
\begin{align*}
\capacity\!\left(B(w_{j-1},\varepsilon),\overline{W}_{j-1}\right)
& \ge \capacity E_G\left(\frac{\diam W_{j-1}}{\varepsilon}\right)\\
& \ge \omega_{n-1}\left(\log\left(\frac{\lambda_n \varepsilon}{\diam W_{j-1}}\right)\right)^{1-n}.
\end{align*}
On the other hand, \eqref{lemma2j} yields
\begin{align*}
      \capacity\!\left(Y_j,\overline{W}_j\right)
& \le \capacity\!\left(B(w_j,\delta),\overline{W}_j\right)\\
& \le \capacity\!\left(B(w_j,\delta),\overline{B}(w_j,\diam W_j)\right)\\
& = \omega_{n-1}\left(\log\left(\frac{\delta}{\diam W_j}\right)\right)^{1-n}.
\end{align*}
Combining the last three estimates we obtain
\[\left(\log\left(\frac{\delta}{\diam W_j}\right)\right)^{1-n}\ge \frac{P_j}{K_I(G)}\left(\log\left(\frac{\lambda_n\varepsilon}{\diam W_{j-1}}\right)\right)^{1-n}\]
for $k_0<j\le k$. Equivalently,
\begin{equation}\label{4i}
\log\left(\frac{\lambda_n\varepsilon}{\diam W_{j-1}}\right)\ge
\left(\frac{P_j}{K_I(G)}\right)^{1/(n-1)}\log \left(\frac{\delta}{\diam W_j}\right).
\end{equation}
Now
\[\log\left(\frac{\lambda_n\varepsilon}{\diam W_{j-1}}\right)=\log\left(\frac{\delta}{\diam W_{j-1}}\right)+\log\frac{\lambda_n \varepsilon}{\delta}.\]
Given $\tau>1$, we may choose $\eta$ so small that $\log(\lambda_n \varepsilon /\delta)\le (\tau-1)\log (\delta/\eta)$. Then
\[\log\left(\frac{\lambda_n\varepsilon}{\diam W_{j-1}}\right)\le \tau \log\left(\frac{\delta}{\diam W_{j-1}}\right)\]
for $k_0<j\le k$. Hence~\eqref{4i} takes the form
\[\log\left(\frac{\delta}{\diam W_{j}}\right)\le \tau\left(\frac{K_I(G)}{P_j}\right)^{1/(n-1)}\log\left(\frac{\delta}{\diam W_{j-1}}\right)\]
for $k_0<j\le k$.
We conclude that
\[\log\left(\frac{\delta}{\diam W_k}\right)
\le \tau^{k-k_0}\left(\frac{K_I(G)^{k-k_0}}{\prod^k_{j=k_0+1}P_j}\right)^{1/(n-1)}
\log\left(\frac{\delta}{\diam W_{k_0}}\right).\]
Using~\eqref{4h} we obtain
\[
\log\left(\frac{\delta}{\diam W_k}\right)
\leq C_0 \left(\frac{\tau^{n-1}K_I(G)^2}{D}\right)^{k/(n-1)}
\log\left(\frac{\delta}{\diam W_{k_0}}\right)
\]
for some constant $C_0$.
With
$$\delta_0=\inf\limits_{1\le \ell\le L_{k_0}}\diam V^{k_0}_\ell
$$
 we obtain
\[
\log\left(\frac{\delta}{\eta}\right)
\le C_0\left(\frac{\tau^{n-1}K_I(G)^2}{D}\right)^{k/(n-1)}
\log\left(\frac{\delta}{\delta_0}\right).
\]
For $\tau$ close to 1 and large $k$ this contradicts~\eqref{4a}.
\end{proof}

\section{Proof of Theorems \ref{T1b}, \ref{T1c}, \ref{T1e}, \ref{T1f}, \ref{T1g}, \ref{T1h} and \ref{T1j}}\label{ProofT1bcefgh}

\begin{proof}[Proof of Theorem \ref{T1b}] Let $f\colon \C\to \C$ be an entire transcendental function and let $J_0(f)$ be the classical Julia set; that is, the set of all $z\in \C$ where the iterates of $f$ do not form a normal family. We refer to~\cite{Bergweiler93} for the basic properties of $J_0(f)$. Let $J(f)$ be as in Definition~\ref{D1a}, with $n=2$ so that $\R^n=\R^2=\C$.

Let $x\in J_0(f)$ and let $U$ be a neighbourhood of~$x$. It is a simple consequence of
Montel's theorem that
$\card\left(\C\backslash O^+_f (U)\right)\le 1$
and hence that
\begin{equation}\label{5a}
\capacity\!\left(\C\backslash O^+_f (U)\right)=0.
\end{equation}
This implies that $x\in J(f)$.

Let now $x\in J(f)$ and let $U$ be a neighbourhood of~$x$. Thus~\eqref{5a} holds. By a result of Baker~\cite{Baker75}, $J_0(f)$ contains continua and thus $\capacity J_0(f)>0$ by Lemma~\ref{lemma2h}.
Therefore $J_0(f)$ is not a subset of $\C\backslash O^+_f (U)$, which means that
${J_0(f)\cap O^+_f (U)\neq \emptyset}$.
 By the complete invariance of $J_0(f)$ we have $J_0(f)\cap U\neq\emptyset$.
As $U$ can be taken arbitrarily small and $J_0(f)$ is closed, we now deduce that $x\in J_0(f)$.
\end{proof}

In order to prove Theorem~\ref{T1c}, we formulate the following result of~\cite{Bergweiler09} already mentioned in the introduction.

\begin{lemma}\label{lemma5a} Let $f$ be an entire \qr map of transcendental type. Then $I(f)$ has at least one unbounded component.
\end{lemma}

\begin{proof}[Proof of Theorem \ref{T1c}] Let $x\in J(f)$.
Then~\eqref{1a1} holds for every neighbourhood $U$ of~$x$.
Since $\capacity BO(f)>0$ by Theorem~\ref{T1d} and $\capacity I(f)>0$ by Lemma~\ref{lemma5a} and Lemma~\ref{lemma2h}, we have
$BO(f)\cap O^+_f(U)\neq \emptyset$ and $I(f)\cap O^+_f (U)\neq \emptyset$.
Since both $BO(f)$ and $I(f)$ are completely invariant, this implies that
$BO(f)\cap U\neq \emptyset$ and
$I(f)\cap  U  \neq \emptyset$.
As $BO(f)\cap I(f) =\emptyset$,  we deduce that
$\partial BO(f)\cap U\neq \emptyset$ and $\partial I(f)\cap  U  \neq \emptyset$.
Since this holds for every neighbourhood $U$ of $x$, we conclude that
$x\in \partial BO(f)$ and $x \in \partial I(f)$.
\end{proof}

\begin{proof}[Proof of Theorem \ref{T1e}] Let $x\in A(\xi)$. Then there exists a neighbourhood $U$ of $x$ such that $f^k |_U\to\xi$ uniformly. Replacing $U$ by a smaller neighbourhood if necessary, we may achieve that
$O^+_f(U)\subset B(0,R)$
for some $R>0$. Thus
$\capacity\!\left(\R^n\backslash O^+_f(U)\right)>0$
which implies that $x\not \in J(f)$.

Let now $x\in J(f)$. Then $x\not\in A(\xi)$ by what we have proved already. Suppose that $x\not\in \partial A(\xi)$. Then there exists a neighbourhood $U$ of $x$ such that $U\cap A(\xi)=\emptyset$. Since $A(\xi)$ is completely invariant this implies that
$O^+_f(U)\cap A(\xi)=\emptyset$ and thus
$A(\xi)\subset \R^n\backslash O^+_f(U)$.
Hence $\capacity A(\xi)=0$ by~\eqref{1a1}, which is a contradiction since $A(\xi)$ is open.
\end{proof}
Theorems~\ref{T1f}, \ref{T1g} and~\ref{T1h} will follow from the next three results.
\begin{theorem}\label{T5a} Let $f$ be an entire \qr map of transcendental type. Suppose that
\begin{equation}\label{5b} \capacity\overline{O^-_f(x)}>0\quad \text{for all}\ x\in \R^n\backslash E(f).
\end{equation}
Then the conclusions {\rm(i)--(iv)} of Theorem~\ref{T1f} hold.
\end{theorem}

\begin{proof} Let $x\in \R^n\backslash E(f)$ and let
 $U$ be an open set intersecting $J(f)$.
It follows from~\eqref{1a1} and~\eqref{5b} that
$O^+_f(U)\cap\overline{O^-_f(x)}\neq \emptyset$.
Since $O^+_f(U)$
is open and $\overline{O^-_f(x)}$ is closed this actually yields that
$O^+_f(U)\cap O^-_f(x)\neq \emptyset$.
This in turn yields both
\begin{equation}\label{5c} x\in O^+_f(U)
\end{equation}
and
\begin{equation}\label{5d} U\cap O^-_f(x)\neq \emptyset.
\end{equation}
Since~\eqref{5d} holds for every open set $U$ intersecting $J(f)$,
we have $J(f)\subset\overline{O^-_f(x)}$ and thus~(i).
Now (ii) follows from (i) since $J(f)$ is closed and
completely invariant.
Since~\eqref{5c} holds for all $x\in \R^n\backslash E(f)$ we obtain
$\R^n\backslash E(f)\subset O^+_f(U)$ which immediately yields~(iii).

To prove (iv) we only have to show that $J(f)$ does not contain isolated
points. Now it follows immediately from (ii) that non-periodic points in $J(f)$ are
not isolated in $J(f)$. Also, $J(f)\backslash E(f)$ contains non-periodic points, since if
$y\in J(f)\backslash E(f)$ is periodic, then all points in
$O^-_f(y)\backslash O^+_f(y)$ are non-periodic.
Thus there exists a non-isolated point $x\in J(f)\backslash E(f)$.
It now follows from (ii) that no point of $J(f)$ is isolated.
\end{proof}
\begin{theorem}\label{T5a1}
Let $f$ be an entire \qr map of transcendental type and let $p\in\N$.
If $\capacity J(f^p)>0$, then $J(f^p)=J(f)$.
\end{theorem}
\begin{proof}
It follows immediately from the definition that  $J(f^p)\subset J(f)$.
In order to prove the reverse inclusion, let $x\in J(f)$ and
let $U$ be a neighbourhood of~$x$.
Since $\capacity\!\left(\R^n\backslash O^+_f(U)\right)=0$ but
$\capacity J(f^p)>0$, there exists $y\in  O^+_f(U)\cap J(f^p)$,
say $y=f^m(z)$ where $m\in\N$ and $z\in U$.
As $V:=f^m(U)$ is a neighbourhood of $y\in J(f^p)$, we have
\begin{equation}\label{5d0}
\capacity\!\left(\R^n\backslash O^+_{f^p}(V)\right)=0.
\end{equation}
We write $m=kp-l$ with $k\in\N$ and
$l\in\{0,1,\dots,p-1\}$.
Then $V= f^{kp}(f^{-l}(U))$ and hence
\[
f^l\!\left(O^+_{f^p}(V)\right)
\subset
f^l\!\left(  O^+_{f^p}\!\left( f^{-l}(U)\right) \right)
=O^+_{f^p}(U).
\]
Noting that $f^l(\R^n)\supset \R^n\backslash E(f^l)$ we deduce that
\begin{equation}\label{5d1}
\begin{aligned}
\R^n\backslash O^+_{f^p}(U)
&\subset \R^n\backslash f^l\!\left(O^+_{f^p}(V)\right)\\
&\subset \left(f^l(\R^n)\backslash f^l\!\left(O^+_{f^p}(V)\right)\right) \cup E(f^l)\\
&\subset f^l\!\left(\R^n\backslash O^+_{f^p}(V)\right) \cup E(f^l).
\end{aligned}
\end{equation}
Lemma~\ref{lemma2i} implies in particular that
quasiregular mappings map sets of capacity zero to sets of
capacity zero.
Thus
$\capacity f^l\!\left(\R^n\backslash O^+_{f^p}(V)\right) =0$
by~\eqref{5d0}.
Since
$E(f^l)$ is finite, we can now deduce from~\eqref{5d1} that
$\capacity\!\left(\R^n\backslash O^+_{f^p}(U)\right)=0$.
Since this holds for every neighbourhood $U$ of $x$,
we conclude that $x\in J(f^p)$.
\end{proof}
\begin{theorem}\label{T5b} Let $f$ be an entire \qr map of transcendental type
which is
locally Lipschitz continuous. Then $\dim\overline{O^-_f(x)}>0$ for all $x\in \R^n\backslash E(f)$.
\end{theorem}

\begin{proof} For functions not having the pits effect, the proof can be carried out in exactly the same way as the proof of Theorem~\ref{T1i},
using part (i) of Lemma~\ref{lemma2d} instead of part~(ii).

For functions with pits effect we proceed as in section~\ref{F w p e} to obtain subsets
$V_1,\dots,V_L$ of $B(0,M)$ with disjoint closures such that $G\colon V_j\to B(0,M)$
is proper for $1\le j\le L$, where $G=g_m$ is conjugate to $f$ by the map
$x\mapsto R_m x$. By hypothesis, there exists $\lambda>0$ such that
\[
|G(x)-G(y)|\le \lambda |x-y|\quad \text{for}\ x,y\in B(0,M).
\]
Lemma~\ref{lemma2d}, (i),  now yields
\[
\dim\overline{O^-_G(x)}\ge \frac{\log L}{\log \lambda}>0
\]
 for all $x\in B(0,M)$. It follows that $\dim\overline{O^-_f(x)}>0$ for all
$x\in B(0,MR_m)$, and as $R_m\to \infty$ this holds for all $x\in \R^n$.
\end{proof}
\begin{proof}[Proof of Theorem \ref{T1f} and Theorem \ref{T1g}] The
conclusions (i)--(iv) of Theorem~\ref{T1f}
follow immediately from
Theorem~\ref{T5b},
Lemma~\ref{lemma2h}
and
Theorem~\ref{T5a}.
To prove  conclusion (v) and Theorem~\ref{T1g}
let $p\in\N$. Since $f^p$ is also
locally Lipschitz continuous we have $J(f^p)=\overline{O^-_{f^p}(x)}$
for all $x\in J(f^p)\backslash E(f^p)$ by~(ii).
Note here that $J(f^p)\backslash E(f^p)\neq \emptyset$ since
$\card J(f^p)=\infty$ by Theorem~\ref{T1a} while $E(f^p)$ is finite.
Thus $\dim J(f^p)>0$ by Theorem~\ref{T5b}. This proves
Theorem~\ref{T1g}.
Moreover, together with Lemma~\ref{lemma2h} this yields $\capacity J(f^p)>0$.
Conclusion (v) of Theorem~\ref{T1f} now follows from Theorem~\ref{T5a1}.
\end{proof}

\begin{proof}[Proof of Theorem \ref{T1h}] The
conclusions (i)--(iv) are immediate consequences of
Theorems~\ref{T1i} and~\ref{T5a}.
In order to prove (v), let $p\in\N$
and put $C=M(1,f^{p-1})$.
Since $f$ does not have the pits effect we see that
if $N\in\N$, then there exist $c>1$ and $\varepsilon>0$
such that for arbitrarily large $R$ the set
\[\left\{x\in \R^n:R\le |x|\le cR,\ |f(x)|\le 1\right\}\]
cannot be covered by $N$ balls of radius $\varepsilon R$.
This implies that
\[\left\{x\in \R^n:R\le |x|\le cR,\ |f^p(x)|\le C\right\}\]
cannot be covered by $N$ such balls.
Thus $g(x):=f^p(Cx)/C$ does not have the pits effect.
By Corollary~\ref{C1a} we have $\capacity J(g)>0$. Since $f^p$ and $g$
are conjugate this implies that $\capacity J(f^p)>0$.
Conclusion (v) now follows from Theorem~\ref{T5a1}.
\end{proof}

\begin{remark}
In the above proof, instead of passing from $f^p$ to $g$ we could also have
used Theorem~\ref{T8a} which implies that the inequality $|f(x)|\leq 1$
in the definition of the pits effect can be replaced by $|f(x)|\leq C$
for any positive constant~$C$.
\end{remark}

The following result will be used to prove Theorem~\ref{T1j}.

\begin{theorem}\label{T5c}
Let $f$ be an entire quasiregular map of transcendental type. If $x\in\R^n\backslash E(f)$ and the local index is bounded on $\overline{O^-_f(x)}$, then $\capacity \overline{O^-_f(x)}>0$.
\end{theorem}
\begin{proof}
Let $x\in\R^n\backslash E(f)$ and suppose that the local index $i(y,f)$ is bounded on $\overline{O^-_f(x)}$. By Theorem~\ref{T1i} we may assume that $f$ has the pits effect. Proceeding as in section~\ref{F w p e}, we find a subset $U\subset B(0,M)$ such that $G\colon U\to B(0,M)$ is a proper map of degree $D$, where $G(y)=g_m(y)=f(R_my)/R_m$ for $y\in U$. By choosing $R_m$ sufficiently large, we may assume that $x\in B(0,MR_m)$ and that
\begin{equation}\label{5e}
D>K_I(f)\max \{i(y,f):y\in\overline{O^-_f(x)}\}.
\end{equation}
Putting $X=\overline{O^-_G(x/R_m)}$, it will suffice to show that $\capacity X>0$.

As in the proof of Theorem~\ref{T1i}, we can use Lemmas~\ref{lemma2d}--\ref{lemma2f} to deduce that $\capacity X>0$ if each $y\in X$ has $P$ pre-images $x_1,\ldots,x_P$ under $G$ satisfying ${|x_i-x_j|\ge\delta}$ for $i\ne j$, where $\delta>0$ and $P$ is the least integer greater than $K_I(f)$. We suppose that this is not the case. Then there exist $\delta_k\to0$, $y_k\in X$ and $x^k_{1},\ldots,x^k_{P-1}\in G^{-1}(y_k)\subset X$ such that
\[ G^{-1}(y_k)\subset \bigcup_{j=1}^{P-1}B(x^k_{j},\delta_k). \]
Passing to a subsequence, we can assume that $y_k\to y_0\in X$ and $x^k_{j}\to x_j\in X$ as $k\to\infty$. It can then be shown that $G^{-1}(y_0)\subset\{x_1,\ldots,x_{P-1}\}$. However, $n(U,y_0,G)=D$ and so there must be some $j\in\{1,\ldots,P-1\}$ for which
\[ i(x_j,G)\ge D/(P-1) \ge D/K_I(f). \]
This leads to a contradiction with \eqref{5e} since $i(x_j,G)=i(R_mx_j,f)$ and $R_mx_j\in \overline{O^-_f(x)}$.
\end{proof}

\begin{proof}[Proof of Theorem \ref{T1j}]
If the local index is bounded on $J(f)$, then applying Theorem~\ref{T5c} to some $x\in J(f)\backslash E(f)$ yields $\capacity \overline{O^-_f(x)}>0$. Thus $\capacity J(f)>0$ by complete invariance. Moreover, arguing as in the proof of Theorem~\ref{T5a} shows that conclusions~(ii) and~(iv) of Theorem~\ref{T1f} hold.

By the complete invariance of $J(f)$ the local index of $f^p$ is
also bounded on $J(f)$. Thus, since always $J(f^p)\subset J(f)$, the local
index of $f^p$ is bounded on $J(f^p)$. Hence $\capacity J(f^p)>0$ by what
we have proved already. Now (v) follows from Theorem~\ref{T5a1}.

If the local index is bounded on $\R^n$, then the conclusion of Theorem~\ref{T1i} holds by Theorem~\ref{T5c}. The conclusions (i)--(iv) of Theorem~\ref{T1f} then follow from Theorem~\ref{T5a}.
\end{proof}

\section{The 2-dimensional case}\label{2dcase}
Let $f\colon \Omega\to \R^n$ be quasiregular.
For $n=2$, the branch set $B_f$ is a discrete subset of $\Omega$.
This is in contrast to the situation for $n\geq 3$, where
the $(n-2)$-dimensional Hausdorff measure of
$f(B_f)$ is positive unless $B_f=\emptyset$; see~\cite[Proposition III.5.3]{Rickman93}.

If $n=2$, the elements of $B_f$ are called \emph{critical points}.
For $x\in B_f$ we call $i(x,f)-1$ the \emph{multiplicity} of the critical point~$x$.
The following result~\cite[\S 1.3]{Steinmetz93} is known as the Riemann-Hurwitz Formula.

\begin{lemma}\label{lemma6a} Let $\Omega_1$ and $\Omega_2$ be domains in $\C$ of connectivity $c_1$ and $c_2$, respectively. Let $f\colon \Omega_1\to \Omega_2$ be a proper \qr map of degree $d$ and denote by $r$ the number of critical points of~$f$, counting multiplicity; that is,
\[r=\sum_{x\in B_f}(i(x,f)-1).\]
Then
\begin{equation}\label{6a} c_1-2=d(c_2-2)+r.
\end{equation}
\end{lemma}
In~\cite{Steinmetz93} it is assumed that $f$ is holomorphic, but the result also holds for \qr maps and in fact for ramified coverings \cite[p.~68]{Milnor06}.
We shall only need the case that $c_1=c_2=1$. Then~\eqref{6a} simplifies to $r=d-1$.
\begin{proof}[Proof of Theorem \ref{T1k}]
By Theorem~\ref{T1h} it suffices to consider functions with pits effect.
Let $f$ be such a function and let $V_1,\dots,V_L$ and $G$ be as in section~\ref{F w p e}.
Thus $G\colon V_j\to B(0,M)$ is a proper map of degree $D_j$ and~\eqref{4a} holds.
Moreover, let also $W^{i,j}_s$ be as in section~\ref{F w p e}. Thus $W^{i,j}_s\subset V_i$
and $G\colon W^{i,j}_s\to V_j$ is a proper map for $1\le i,j \le L$ and
$1\le s\le t_{i,j}$.
We may assume that the $V_j$ and $W^{i,j}_s$ are connected, as this
is the case if $|G(c)|\neq M$ and $|G^2(c)|\neq M$ for all critical
points $c$ of~$G$, and hence can be achieved by perturbing $M$ slightly.
By the maximum principle, the $V_j$ and $W^{i,j}_s$ are in
fact simply connected. By the Riemann-Hurwitz Formula
(Lemma~\ref{lemma6a}), $W^{i,j}_s$ contains
$\deg(G\colon W_s^{i,j}\to V_j)-1$ critical points, counting multiplicities.

Let
\[Y_j=G^{-1}(V_j)\cap B(0,M)=\bigcup^L_{i=1}\bigcup^{t_{i,j}}_{s=1}W_s^{i,j}\]
and denote by $\mu_j$ the number of critical points of $G$ in $Y_j$. Then
\begin{align*}
\mu_j &=\sum^L_{i=1}\sum^{t_{i,j}}_{s=1}
\left(\deg\!\left(G\colon W^{i,j}_s\to V_j\right)-1\right)\\
  &=\sum^L_{i=1}\sum^{t_{i,j}}_{s=1}
\deg\!\left(G\colon W^{i,j}_s\to V_j\right)-\sum^L_{i=1}t_{i,j}
\\ &
   =D-\sum^L_{i=1}t_{i,j}
\end{align*}
by \eqref{4a0}. Hence
\[LD-\sum^L_{j=1}\sum^{L}_{i=1}t_{i,j}=\sum^L_{j=1}\mu_j\le D.\]
Defining $L_2$ and $V_1^2,\dots,V^2_{L_2}$ as in section~\ref{F w p e} we obtain
\[L_2=\sum^L_{j=1}\sum^{L}_{i=1}t_{i,j}\ge (L-1)D\ge D\]
and thus
\begin{equation}\label{6b} L_2>K_I(G)^2\ge K_I(G^2)
\end{equation}
by~\eqref{4a}. Recall that $G^2\colon V^2_j\to B(0,M)$ is a proper map and
$\overline{V}^2_j\subset B(0,M)$ for $1\le j\le L_2$. Similarly as before
we can now deduce from~\eqref{6b}, Lemma~\ref{lemma2d} and Lemma~\ref{lemma2f}
that $\capacity\overline{O^-_{G^2}(x)}>0$ for all $x\in B(0,M)$.
This implies that $\capacity\overline{O^-_f(x)}>0$ for all $x\in \R^2$; that is,
the conclusion of Theorem~\ref{T1i} holds.
Theorem~\ref{T5a} now implies that the conclusions (i)--(iv) of Theorem~\ref{T1f} holds.
Using part~(ii) of this theorem we conclude that $\capacity J(f)>0$.

We may apply this reasoning also to the iterates of $f$. Thus
$\capacity J(f^p)>0$ for all $p\in\N$.
Conclusion (v) of Theorem~\ref{T1f} now follows from Theorem~\ref{T5a1}.
\end{proof}

\section{Examples}\label{Examples}
\begin{example}\label{Ex7e}
In the iteration theory of transcendental entire functions, much attention
has been paid to the exponential functions $E_\lambda(z)=\lambda e^z$.
Here we only mention a result of Devaney and Krych~\cite[p.~50]{Devaney84}
saying that if $0<\lambda<1/e$, then
$J(E_\lambda)$ is
the complement
of the attracting basin of the attracting fixed point of~$E_\lambda$
and has the structure of a
{\em Cantor bouquet}.
By definition,
this is a union of uncountably many pairwise disjoint
simple curves connecting finite points in $\C$ (or $\R^n$)
with infinity.
For further results on exponential dynamics
we refer to a detailed survey by Devaney~\cite{Devaney}, as well as papers
by Rempe~\cite{Rempe06} and Schleicher~\cite{Schleicher03a}.

Zorich \cite{Zorich67} introduced
transcendental type quasiregular mappings that are \mbox{$3$-dimensional} analogues of the exponential function.
It was shown in~\cite{Bergweiler10a} that, for a suitable choice of
parameters, Zorich maps also have attracting fixed points whose attracting
basins are complements of Cantor bouquets.
Moreover, results of Karpi\'nska~\cite{Karpinska99} and McMullen~\cite{McMullen87}
concerning the Cantor bouquets of exponential functions
were extended to Zorich maps.
Here we show that -- in analogy to the result of Devaney and Krych --
 these $3$-dimensional  Cantor bouquets coincide
with the Julia set as defined in Definition~\ref{D1a}.

To define a Zorich map, we follow \cite[\S6.5.4]{Iwaniec01} and consider the square $Q = [-1,1]^2$ and the upper hemisphere
\[ U=\{x=(x_1,x_2,x_3)\in\R^3: |x|=1,\, x_3\ge0\}. \]
Let $h\colon Q\to U$ be a bi-Lipschitz map, and define
\[ F\colon Q\times\R\to\R^3, \qquad F(x_1,x_2,x_3)=e^{x_3}h(x_1,x_2). \]
Then $F$ maps the infinite square beam $Q\times\R$ to the upper half-space. Repeated reflection along sides of square beams and the $(x_1,x_2)$-plane yields a map ${F\colon \R^3\to\R^3}$. This map $F$ is quasiregular, omits the value $0$ and is doubly-periodic in the $x_1$- and $x_2$-directions. We call a function $F$ defined this way a {\em Zorich map} and we apply this term also to functions $f_a$ given by
\[ f_a\colon \R^3\to\R^3, \qquad f_a(x)=F(x)-(0,0,a). \]

For a Zorich map $f_a$ as above with parameter $a$ chosen sufficiently large,
it was proved in \cite{Bergweiler10a} that there exists a unique
attracting fixed point $\xi$ of $f_a$ such that the set
\[ J_0:= \R^3\backslash A(\xi) = \{ x\in\R^3:f_a^k\not\to\xi\} \]
is a Cantor bouquet.
As mentioned, we want to show that
\begin{equation}\label{7a}
J(f_a)=J_0.
\end{equation}
By Theorem~\ref{T1e} we have $J(f_a)\subset J_0$ so that we only have
to prove the reverse inclusion.

We will use the following notation. For $r=(r_1,r_2)\in\Z^2$, we put
\[ P(r)= P(r_1,r_2)=\{(x_1,x_2)\in\R^2 : |x_1-2r_1|<1, |x_2-2r_2|<1\} \]
so that $P(0,0)$ is the interior of $Q$. For $c\in\R$, we denote the half-space $\{(x_1,x_2,x_3)\in\R^3:x_3>c\}$ by $H_{>c}$, and we define $H_{\ge c}$, $H_{<c}$ similarly. Observe that $f_a$ maps $P(r_1,r_2)\times\R$ bijectively onto $H_{>-a}$ if $r_1+r_2$ is even and bijectively onto $H_{<-a}$ if $r_1+r_2$ is odd. Let $S=\{(r_1,r_2)\in\Z^2:r_1+r_2 \mbox{ is even}\}$.

In \cite{Bergweiler10a}, constants $M\in\R$ and $\alpha\in(0,1)$ are found such that, for any $r\in S$, there exists a branch of the inverse function of $f_a$,
\[ \Lambda^r\colon H_{\ge M}\to P(r)\times(M,\infty) =: T(r), \]
that satisfies \cite[(2.3)]{Bergweiler10a}
\begin{equation}\label{7b}
|\Lambda^r(x)-\Lambda^r(y)| \le \alpha|x-y| \quad \mbox{for } x,y\in H_{\ge M}.
\end{equation}
This estimate leads to the following uniform expansion property of $f_a$ on $\Lambda^r(H_{\ge M})$.
\begin{lemma}\label{lemma7a}
If $x\in \Lambda^r(H_{\ge M})$ for some $r\in S$ and if $R>0$, then
\[ f_a\left(B(x,R)\cap H_{\ge M}\right) \supset B\left(f_a(x),\alpha^{-1}R\right)\cap H_{\ge M}. \]
\end{lemma}
\begin{proof}
Take $y\in B\left(f_a(x),\alpha^{-1}R\right)\cap H_{\ge M}$. Then by \eqref{7b}
\[ |x-\Lambda^r(y)| = |\Lambda^r(f_a(x))-\Lambda^r(y)|
\le \alpha|f_a(x)-y|<R. \]
Hence $\Lambda^r(y)\in B(x,R)\cap T(r)$ and so $y=f_a(\Lambda^r(y))\in f_a(B(x,R)\cap H_{\ge M})$.
\end{proof}
We shall also require the fact that \cite[p. 608]{Bergweiler10a}
\begin{equation}\label{7c}
J_0 \subset\bigcup_{r\in S}T(r).
\end{equation}

We are now ready to establish \eqref{7a} by showing that $J_0\subset J(f_a)$. To this end, take $x_0\in J_0$ and let $U$ be a neighbourhood of $x_0$. Write $x_k=(x_{k,1},x_{k,2},x_{k,3}):=f_a^k(x_0)$ and note that $x_k\in\bigcup_{r\in S}T(r)\subset H_{\ge M}$, for all $k\ge0$, due to \eqref{7c} and the complete invariance of $J_0$. It follows that we may repeatedly apply Lemma~\ref{lemma7a} to obtain a sequence $R_k\to\infty$ such that
\[ f_a^k(U)\supset B(x_k,R_k)\cap H_{\ge M}. \]
Provided that $k$ is large, we can find $(p_{k,1},p_{k,2})\in\Z^2$ such that the set
\[ V_k=\{(y_1,y_2)\in\R^2:|y_1-2p_{k,1}|\le2, |y_2-2p_{k,2}|\le2\} \times [x_{k,3}, x_{k,3}+R_k/2] \]
is contained in $B(x_k,R_k)\cap H_{\ge M}$.
Note that $f_a$ maps $V_k$ onto the shell
\[ A_k=\{y\in\R^3: e^{x_{k,3}}\le|y+(0,0,a)|\le e^{x_{k,3}+R_k/2} \} \]
and therefore we have that $A_k\subset f_a^{k+1}(U)$. Since $x_{k,3}\ge M$ and $R_k\to\infty$, it is not difficult to see that for large $k$ we can find $(q_{k,1},q_{k,2})\in\Z^2$ and $t_k>0$ such that $t_k\to\infty$ and
\[ \{(y_1,y_2,y_3)\in\R^3:|y_1-2q_{k,1}|\le2, |y_2-2q_{k,2}|\le2, |y_3|\le t_k\} \subset A_k. \]
By considering the image of this set, we now deduce that
\[ \{y\in\R^3: e^{-t_k}\le|y+(0,0,a)|\le e^{t_k} \} \subset f_a(A_k) \subset f_a^{k+2}(U). \]
This implies that $O^+_{f_a}(U)=\R^3\backslash\{(0,0,-a)\}$ and thus $x_0\in J(f_a)$, completing the proof of \eqref{7a}.
\end{example}
\begin{example}\label{Ex7d}
Quasiregular mappings of $\R^n$ that can be seen as analogues of the trigonometric functions were constructed in \cite{Bergweiler11} as follows. Write ${x=(x_1,\ldots,x_n)}$ for points in $\R^n$. Let $F$ be a bi-Lipschitz map from the half-cube $[-1,1]^{n-1}\times[0,1]$ to the upper half-ball
\[ \{x\in\R^n : |x|\le1, x_n\ge0\} \]
which maps the face $[-1,1]^{n-1}\times\{1\}$ onto the hemisphere
\[ \{x\in\R^n:|x|=1,x_n\ge0\}.\]
Extend $F$ to a mapping $F\colon [-1,1]^{n-1}\times[0,\infty)\to\{x\in\R^n:x_n\ge0\}$ by
\[ F(x) = e^{x_n-1}F(x_1,\ldots,x_{n-1},1), \quad \mbox{for } x_n>1. \]
Then $F$ bijectively maps a half-infinite square beam onto the upper half-space. Using repeated reflections in hyperplanes, $F$ is extended to give a quasiregular self-map of $\R^n$; see \cite{Bergweiler11} for more details. This construction quickly leads to the fact that, for large enough $\lambda>0$, the map $f:=\lambda F$ is locally uniformly expanding. Choosing $F$ so that it fixes the origin, and taking $\lambda$ sufficiently large, this expansion property was used in \cite{FletcherSine} to prove that $O^+_f(U)=\R^n$ for all non-empty open subsets $U\subset\R^n$. Thus $J(f)=\R^n$. Furthermore, the periodic points of $f$ were shown to form a dense subset of $\R^n$.
\end{example}

\begin{example}\label{Ex7a}
Define $g\colon \C\to\C$ by $g(z)=z+1+e^{-z}$ and, for a large positive constant $M$, let
\[ f(z) = \begin{cases} g(z) &\mbox{if } \Real z\le M \mbox{ or } \Real z\ge 2M, \\
g(z)+(1+e^{-z})\sin\left(\dfrac{\pi\Real z}{M}\right) &\mbox{if } M<\Real z <2M. \end{cases} \]
It is easy to see that $f$ is quasiregular of transcendental type if $M$ is large. In fact, $K(f)\to1$ as $M\to\infty$.

The function $g$ is a classical example considered by Fatou \cite[Exemple I]{Fatou26} who showed that with the right half-plane $H=\{z:\Real z>0\}$ we have $g(H)\subset H$ and $g^k|_H\to\infty$ as $k\to\infty$. We also have $f(H)\subset H$ which implies that $J(f)\cap H=\emptyset$. With $\xi=3M/2$, we have $f(\xi)=\xi$ and $f(x)>x$ for $x>\xi$. Thus $f^k(x)\to\infty$ as $k\to\infty$ for $x>\xi$. We conclude that $(\xi,\infty)\subset I(f)$, while $\xi\in BO(f)$. Hence $\xi\in(\partial BO(f) \cap \partial I(f))\backslash J(f)$.
\end{example}

\begin{example}\label{Ex7b}
The quasiregular map $\tilde{f}\colon \C\to\C$ constructed in \cite[\S 4]{Nicks} is of transcendental type and has the following properties.

The upper half-plane $H_+=\{z:\Imag z>0\}$ is mapped into itself by $\tilde{f}$ and hence $J(\tilde{f})\cap H_+ = \emptyset$.
There is a sequence of domains $(W_k)_{k\in\Z}$ with closures in $H_+$ such that $\tilde{f}(W_k)=W_{k+1}$ and $\tilde{f}(z)=z/2$ on each ${W_k}$. Therefore, the iterates of $\tilde{f}$ converge to $0$ on ${W_k}$ and so ${W_k}\subset BO(\tilde{f})$ for $k\in\Z$. In contrast, all points of $H_+$ that are not contained in some $\overline{W_k}$ escape to infinity under iteration; that is,
\[ H_+\backslash\bigcup_{k\in\Z}\overline{W_k}\subset I(\tilde{f}). \]
 It can then be shown that $\partial W_k\subset(\partial BO(\tilde{f}) \cap \partial I(\tilde{f}))\backslash J(\tilde{f})$ for each $k\in\Z$.
\end{example}

\begin{example}\label{Ex7c}
Let $\delta>0$, define $g\colon [0,3]\to\R$ by
\[ g(r) = \begin{cases} (1-\delta) &\mbox{if }0\le r\le1, \\
1+\delta(r-2) &\mbox{if }1<r\le2,\\
1 &\mbox{if }2<r\le3, \end{cases} \]
and define $f\colon \C\to\C$ by
\[ f(z) = \begin{cases} zg(|z|) &\mbox{if }|z|\le3, \\
z+\delta(|z|-3)e^z &\mbox{if }3<|z|\le4,\\
z+\delta e^z &\mbox{if }4<|z|, \end{cases} \]
(cf.~\cite[Example~5.3]{Bergweiler10}). If $\delta$ is small enough then $f$ is quasiregular of transcendental type, and in fact $K(f)\to1$ as $\delta\to0$.

The function $f$ has an attracting fixed point at $0$ and it is easily seen that $B(0,2)\subset A(0)$ and $S(0,2)\subset \partial A(0)$. On the other hand, $f(B(0,3))=B(0,3)$ and thus $S(0,2)\cap J(f)=\emptyset$. We conclude that $\partial A(0)\not\subset J(f)$.
\end{example}

\section{Some consequences of Harnack's inequality}\label{Harnack}
We show that in the condition $|f(x)|\leq 1$ appearing in the
definition of the pits effect (Definition~\ref{D1a}) the constant $1$
can be replaced by any other positive constant.
In fact, we have the following result.
\begin{theorem}\label{T8a}
Let $f\colon \R^n\to \R^n$ be a \qr map of transcendental type
having the pits effect and let $\alpha>1$.
Then there exists $N\in \N$ such that for all $\alpha>1$,
all $c>1$ and
all $\varepsilon>0$ there exists $R_0$ such that if $R>R_0$, then
\[
\left\{x\in \R^n:R\le |x|\le cR,\ |f(x)|\le R^\alpha\right\}
\]
can be covered by $N$ balls of radius $\varepsilon R$.
\end{theorem}
The following result is known as Harnack's inequality; see \cite[Theorem VI.7.4, Corollary VI.2.8]{Rickman93}.
\begin{lemma}\label{lemma8a}
For $K\geq 1$ and $n\geq 2$ there exists a constant $\theta>1$ such that
if $ f\colon B^n(a,2r)\to \R^n$ is $K$-quasiregular and $|f(x)|>1$ for all
$x\in B^n(a,2r)$, then
\[
\sup_{x\in B^n(a,r)}\log |f(x)|\le \theta\inf_{x\in B^n(a,r)}\log |f(x)|.
\]
\end{lemma}
A simple consequence is the following result.
\begin{lemma}\label{lemma8b}
Let $\Omega\subset \R^n$ be a domain, $C$ a compact subset
of $\Omega$ and $K\geq 1$. Then there exists $\beta>1$ such that
if $ f\colon \Omega \to \R^n$ is $K$-\qr and $|f(x)|>1$ for all
$x\in \Omega$, then
\[
\max_{x\in C}\log |f(x)|\le \beta\min_{x\in C}\log |f(x)|.
\]
\end{lemma}

\begin{proof}[Proof of Theorem \ref{T8a}]
Let $N$ be as in Definition~\ref{D1b} (i.e., the definition of the pits effect)
and let $c,\alpha>1$ and $\varepsilon>0$.
Suppose that the conclusion is false. Then
there exists a sequence $(R_m)$ tending to infinity such that
\begin{equation}\label{8a}
\{x\in \R^n:R_m\le |x|\le cR_m,\ |f(x)|\le R_m^\alpha\}
\end{equation}
cannot be covered by $N$ balls of radius $\varepsilon R_m$.
With $h_m(x)=f(R_m x)$ as in~\eqref{3a} this means that
\[
\left\{x\in \R^n:1\le |x|\le c,\ |h_m(x)|\le R_m^\alpha\right\}
\]
cannot be covered by $N$ balls of radius $\varepsilon$.
On the other hand, for any $\delta>0$ there exists $R_0$ such that if $R> R_0$, then
$\{x\in \R^n: R/2\le |x|\le 2cR, |f(x)|\le 1\}$
 can be covered by $N$ balls of radius $\delta R$.
In terms of $h_m$ this means that if $R_m> R_0$, then
$\{x\in \R^n: 1/2\le |x|\le 2c, |h_m(x)|\le 1\}$
can be covered by  $N$ balls of radius~$\delta$, say
\[
\left\{x\in \R^n: \frac12\le |x|\le 2c, |h_m(x)|\le 1\right\}
\subset\bigcup_{j=1}^N B(x_{m,j},\delta).
\]
We may assume that $(x_{m,j})_{m\in\N}$ converges for all
$j\in\{1,\dots,N\}$, say
$x_{m,j}\to x_j$ as $m\to\infty$.
Then
\[
\left\{x\in \R^n: \frac12\le |x|\le 2c, |h_m(x)|\le 1\right\}
\subset\bigcup_{j=1}^N B(x_{j},2\delta)
\]
for large~$m$.
We would like to apply Lemma~\ref{lemma8b} to
$$\Omega= \left\{x\in \R^n:\frac12<|x|<2c \right\}\backslash\bigcup_{j=1}^N
\overline{B}(x_j,2\delta)$$
and the subset
$$C= \left\{x\in \R^n:1\leq |x|\leq c \right\}\backslash
\bigcup_{j=1}^N B(x_j,5N\delta).$$
However, the set $\Omega$ defined this way need not be connected.
But choosing $\delta$ sufficiently small we can achieve that
there exists $t\in[1,c]$ such that the sphere $S(0,t)$
is contained in $C$
while all components of $\Omega$  that do not contain
$S(0,t)$ have diameter less than $4N\delta$.
Let $\Omega'$ be the component of $\Omega$ containing~$S(0,t)$.
Then
$C$
is a compact subset of $\Omega'$ and thus
Lemma~\ref{lemma8b} can be applied to $\Omega'$ and~$C$.
We obtain
\[
\beta \min_{x\in C}\log |h_m(x)|\ge
\max_{x\in C}\log |h_m(x)|\ge\log M(t,h_m)=
\log M(R_mt,f).
\]
Lemma~\ref{lemma2a} yields
\[
\min_{x\in C}\log |h_m(x)|> \alpha\log R_m
\]
for large $m$, which means that
\[
|f(x)|> R_m^\alpha \quad\text{for} \ R_m\leq |x|\leq cR_m,\  x\notin
\bigcup_{j=1}^N B(R_mx_j,5N\delta R_m).
\]
As we may assume that $5N\delta<\varepsilon$, we see that the set given
by~\eqref{8a} can be covered by $N$ balls of radius $\varepsilon R_m$.
This is a contradiction.
\end{proof}

\begin{remark}\label{R8a}
At the beginning of section~\ref{F w p e} we remarked that the functions
$g_m$ defined by~\eqref{3b} form a quasinormal family if $f$ has the
pits effect.
In order to see this, we note that by Theorem~\ref{T8a} we can  cover
\[
\left\{x\in \R^n:R_m\le |x|\le cR_m,\ |f(x)|\le R_m^\alpha\right\}
\]
by $N$ balls of radius $\varepsilon R_m$, provided $m$ is large.
Similarly, we can cover
\[
\left\{x\in \R^n:R_m/c\le |x|\le R_m,\ |f(x)|\le R_m^\alpha/c^\alpha\right\}
\]
by $N$ such balls. Choosing $c=1/\varepsilon$ we can thus cover
\[
B(0,1/\varepsilon)\backslash
\left\{x\in \R^n: |g_m(x)|\le \varepsilon^\alpha R_m^{\alpha-1}\right\}
\]
by $2N+1$ balls of radius $\varepsilon$, provided $m$ is large.
As this holds for every $\varepsilon>0$, and since $\alpha>1$,
we easily see that every subsequence of $(g_m)$ has a subsequence
which converges to $\infty$ in $\R^n\backslash E$ for some set
$E$ of cardinality at most $2N+1$.
Thus $(g_m)$ is quasinormal.
\end{remark}

\end{document}